\documentclass[11pt]{amsart}

\usepackage{amsmath,amsfonts,amssymb,amsthm}
\usepackage{txfonts}
\usepackage{latexsym,bm,graphicx}
\usepackage{mathrsfs}
\usepackage{color}

\title[Weyl's law on mm-spaces]{Weyl's law on $RCD^*(K,N)$ metric measure spaces}

\author{Hui-Chun Zhang}
\address{Hui-Chun Zhang: Department of Mathematics, Sun Yat-sen University, No. 135, Xingang Xi Road, Guangzhou, 510275 \\}
\email{zhanghc3@mail.sysu.edu.cn}
\author{Xi-Ping Zhu}
\address{Xi-Ping Zhu: Department of Mathematics, Sun Yat-sen University,  No. 135, Xingang Xi Road, Guangzhou, 510275 \\}
\email{stszxp@mail.sysu.edu.cn}

\usepackage[english]{babel}
\usepackage[utf8]{inputenc}

\usepackage{amsmath,amsfonts,amssymb,amsthm}
\usepackage{mathrsfs}

\newtheorem{thm}{Theorem}[section]
\newtheorem{prop}[thm]{Proposition}
\newtheorem{lem}[thm]{Lemma}

\newtheorem{cor}[thm]{Corollary}

 \theoremstyle{definition}
 \theoremstyle{remark}

\newtheorem{defn}[thm]{Definition}
\newtheorem{rem}[thm]{Remark}
\newtheorem{exam}[thm]{Example}

\numberwithin{equation}{section}

\newcommand{\ls}{\leqslant}
\newcommand{\gs}{\geqslant}
\newcommand{\du}{{\rm d}\mu}

\newcommand{\ip}[2]{\langle{#1},{#2}\rangle}

\begin{document}



\begin{abstract}
In this paper, we will prove the Weyl's law for the asymptotic formula of  Dirichlet eigenvalues on metric measure spaces with generalized Ricci curvature bounded from below.
\end{abstract}

\maketitle

\tableofcontents

\section{Introduction}
One of most fundamental theorems in spectral geometry is the Weyl's law \cite{dav89}, which
states that, on any closed $n$-dimensional Riemannian manifold $(M^n,g)$, we have a leading asymptotic
$$\lim_{\lambda\to\infty}\frac{N(\lambda) }{\lambda^{n/2}}= \frac{\omega_n}{(2\pi)^n}vol_g(M^n),$$
where $\lambda_j$, $1\ls j<\infty$, are the eigenvalues of Laplace-Beltrami operator $\Delta$ on $(M^n,g)$, and $N(\lambda)$ is the spectral counting function
$$N(\lambda) := \#\{\lambda_j \in {\rm Spec}(\Delta),\  \lambda_j \ls \lambda\},$$
and $\omega_n$ is the volume of unit ball in $\mathbb R^n$, and $vol_g(M^n)$ is the volume of $M^n$. If $\Omega\subset M^n$ is a bounded domain in  $(M^n,g)$  with smooth boundary, then the same asymptotic formula holds for  the Dirichlet (or Neumann) eigenvalues, by replacing $vol_g(M^n)$ by $vol_g(\Omega).$

It has a wide range of interests about the extensions of  Weyl's law, (see, for examples, \cite{mor08,mu07,mil16} and a survey \cite{ivr16}). In particular, on a weighted Riemannian manifold with Bakry-Emery Ricci curvature bounded from below, if the density $\mu:=f\cdot vol_g$ is smooth, and is bounded away from $0$ and $\infty$, then it was shown by E. Milman in \cite{mil16} that the classical Weyl's law still holds for weighted Laplacian $\Delta_\mu:=\Delta +\ip{\nabla \ln f}{\nabla}$.

 In this paper, we will extend this classical result to  \emph{non-smooth}  settings.
To formulate our main result, we need to introduce some notations. Let $(X,d,\mu)$ be a metric measure space (a metric space equipped  a Radon measure).  A synthetic notion of lower Ricci bounds on $(X,d,\mu)$ was introduced in the pioneering works of Sturm \cite{stu06-1,stu06-2} and Lott-Villani \cite{lv09,lv07-jfa}. Nowadays,  many important developments were given in this field (see \cite{ags14,ags15,ams16,ags-duke,bs10,eks15,cm16,gig13,mw16,hkx13,jia15} and so on). In particular, to rule out the Finsler spaces, an improvement notion, $RCD(K,\infty)$-condition, was introduced by Ambrosio-Gigli-Savar\'e in \cite[\S 5]{ags-duke}. The finitely dimensional case, $RCD(K,N)$, was given  by  Gigli in  \cite[\S 4.3]{gig13,gig15}, and a splitting theorem for $RCD(0,N)$-space was proved by Gigli \cite{gig13}. The parameters $K$ and $N$   play the role of  ``Ricci curvature $\gs K$ and dimension $\ls N$".
Very recently,  Ambrosio-Gigli-Savar\'e \cite{ags14},  Erbar-Kuwada-Sturm \cite{eks15} and Ambrosio-Mondino-Savar\'e \cite{ams16}  introduced a  Bakry-Emery condition $BE$, which is a weak formulation of  Bochner inequality. They proved in \cite{ams16,eks15} that the condition $BE(K,N)$ is equivalent to the  (reduced) Riemannian curvature-dimension condition  $RCD^*(K,N)$ for constants $K\in\mathbb R$ and $N\gs1$.  In \cite[Theorem 1.1]{cmil16}, Cavalletti-Milman showed that the condition $RCD^*(K,N)$ is equivalent to the condition $RCD(K,N)$ provided  the total measure $\mu(X)<\infty$.

Let $(X,d,\mu)$ be a metric measure space satisfying $RCD^*(K,N)$  for some $K\in \mathbb R$ and $N\in[1,\infty)$. For any bounded domain $\Omega\subset X$, according to \cite{che99,shan00,ags14}, the Sobolev spaces $W^{1,p}(\Omega)$, $1\ls p\ls\infty$, are well defined. Moreover,   the  space $W^{1,2}(\Omega)$ is a Hilbert space (\cite{ags14,gig15}).  The Cheeger energy over $\Omega$
$${\rm Ch}(f)=\int_\Omega|\nabla f|^2\du$$
provides a closed quadratic form acting on the Sobolev space  $W_0^{1,2}(\Omega)$, where $|\nabla f|$ is the  weak upper gradient of $f$ (\cite{ags14}). The Dirichlet form $({\rm Ch}, W_0^{1,2}(\Omega))$ is associated with a self-adjoint operator $\Delta_{\Omega}$.
    If ${\rm diam}(\Omega)\ls {\rm diam}(X)/a$ for some $a>1$,  then the Rellich's compactness theorem holds (see \cite{bm06,gms15,hk00}), and hence the operator $(Id-\Delta_{\Omega})^{-1}$ is compact. The classical spectral theorem  implies that  Dirichlet spectrum is discrete,  denoted by
$$0<\lambda_1^{\Omega}\ls \lambda_2^{\Omega}\ls \cdots\ls \lambda_m^{\Omega}\ls \cdots, \quad j\in \mathbb N.$$
Our main result in this paper is the following Weyl asymptotic formula for these  Dirichlet eigenvalues:
\begin{thm}\label{thm1.1}
 Let $ (X,d,\mu)$ be a metric measure space satisfying $RCD^*(K,N)$ for some $K\in\mathbb R$ and some $N\gs1$.  Suppose that  the measure $\mu$ and the $N$-dimensional Hausdorff dimension $\mathscr H^N$ are mutually  absolutely continuous. Namely, $\mu\ll \mathscr H^N\ll\mu$.  Let $\Omega\subset X$ be a bounded domain of $X$
such  that ${\rm diam}(\Omega)\ls {\rm diam}(X)/s$  for some $s>1$.
 Then  $N$ is an integer  and it holds the asymptotic formula:
  \begin{equation}\label{equa1.1}
 \lim_{\lambda\to\infty}\frac{N_\Omega(\lambda) }{\lambda^{N/2}}=  \frac{\omega_N\cdot \mathscr H^N(\Omega)}{(2\pi)^{N}},
 \end{equation}
 where  $N_\Omega(\lambda):=\#\{\lambda_j^\Omega:\ \lambda_j^\Omega\ls \lambda\}.$
 \end{thm}

Remark that the RHS of (\ref{equa1.1}) does  not depend on the measure $\mu$. Theorem \ref{thm1.1} is a consequence of Theorem \ref{thm4.6},   a more general result on $RCD^*$-spaces.
 In the case of a smooth Riemannian manifold $(M,g)$ of $n$-dimension with the Riemannian volume $\mu:=vol_g$, the relation (\ref{equa1.1}) recovers the classical Weyl's law.

Let us look at the case of an $n (\gs2)$-dimensional Alexandrov space $(X,d)$ with the Hausdorff measure $\mathscr H^n$, and with curvature $\gs k$ for some $k\in\mathbb R$. It was proved \cite{pet11,zz10} that $(X,d,\mathscr H^n)$ satisfies $RCD^*((n-1)k,n)$.
 From Theorem \ref{thm1.1}, we have the following consequence.
\begin{cor}\label{cor1.2}
Let $\Omega$ be a bounded domain in an $n$-dimensional Alexandrov space $(X,d,\mathscr H^n)$. Then we have the Weyl's law
 \begin{equation}\label{equa1.2}
 \lim_{\lambda\to\infty}\frac{N_\Omega(\lambda) }{\lambda^{n/2}}=  \frac{\omega_n\cdot \mathscr H^n(\Omega)}{(2\pi)^{n}}.
 \end{equation}
\end{cor}

Another consequence is that the Weyl's law also holds for noncollapsing limit spaces in the sense of Cheeger-Colding. More precisely, if $(X,d,\mu)$ is a measured Gromov-Hausdorff limit space of a sequence of pointed Riemannian manifolds $(M_j,g_j,p_j)$ with
$$Ric_{M_j}\gs K,\quad \ \dim(M_j)=n,\quad \ vol_{g_j}(B_1(p_j))\gs v_0>0,$$
then the Weyl's law (\ref{equa1.2}) still holds. This case has been already proved by Ding in \cite{ding02}.

Recalling that in the  proof of the Weyl's law on smooth setting, a key ingredient is   a uniformly small time asymptotic behaviour of heat trace $H(t,x,x)$ via the parametrix of heat kernels. However, the construction of the parametrix on smooth manifolds does not work on  \emph{singular} metric measure spaces. To deal with this lack of the parametrix, we shall get the small time asymptotic behavior via  the (locally) uniform convergence of Dirichlet heat kernels living on a converging sequence of metric measure spaces in the sense of  pointed measured Gromov-Hausdorff topology,  as in \cite{ding02,xu14,gms15}.

As a byproduct, we  show a local spectral convergence on $RCD^*(K,N)$-spaces, which is of independent interesting (See Theorem \ref{thm3.8}).
\begin{prop}\label{prop1.3}
  Let pointed metric measure spaces  $(X_j,d_j,\mu_j,p_j)_{j\in\mathbb N}$ converge to $(X_\infty,d_\infty,\mu_\infty,p_\infty)$ in the sense of pointed measured Gromov-Hausdorff.  Suppose that all $(X_j,d_j,\mu_j)$ satisfy $RCD^*(K,N)$ for some $K\in \mathbb R$ and some $N\gs1$.   Let  $R>0$ with  $R\in(0,{\rm diam}(X_j)/a)$ for some $a>2$, $\forall j\in\mathbb N$. Assume that $\partial B_R(p_\infty)=\partial \big(X_\infty\backslash \overline{B_R(p_\infty)}\big)$.\footnote{We remark that this assumption can be replaced by ${\rm Cap}_2\big(\partial B_R(p_\infty)\backslash \partial \big(X_\infty\backslash \overline{B_R(p_\infty)}\big)\big)=0.$ }

For each $j\in\mathbb N$,  we denote by $\lambda_{m,j}^{(R)}$  the $m-$th Dirichlet eigenvalues of $\Delta_{B_R(p_j)}$ on ball $B_R(p_j)$.
Then we have that the spectral convergence
$$\lim_{j\to\infty}\lambda_{m,j}^{(R)}=\lambda^{(R)}_{m,\infty}.$$
\end{prop}
\begin{rem}\label{remark1.4}
 A spectral convergence theorem for eigenvalues $\lambda_{m,j}(X_j)$, different from the local Dirichlet eigenvalues in Proposition \ref{prop1.3}, on a sequence of convergent compact metric measure spaces $(X_j,d_j,\mu_j)$ was proved by Gigli e.t.  in \cite{gms15}.
\end{rem}

The following example shows that the assumption
$$\partial B_R(p_\infty)=\partial \big(X_\infty\backslash \overline{B_R(p_\infty)}\big)$$
is necessary in  Proposition \ref{prop1.3}. We would like to thank Prof. S. Honda  for telling us such an example.

\begin{exam}\label{example1.5}
 Let $X_j:=[-1+\frac{1}{j},1-\frac{1}{j}]$ equip the Euclidean distance $d_E$ and the 1-dimensional Lebesgue measure $\mathcal L^1$ and let $p_j=0$. Then we have
$$\big(X_j,d_E,\mathcal L^1,p_j\big)\overset {pmGH}{\longrightarrow}\big(X_\infty:=[-1,1], d_E,\mathcal L^1,p_\infty:=0\big).$$
 Now we consider the balls $B_1(p_j)\ (=X_j)$. It is clear that $B_1(p_j)$ converge to $B_1(p_\infty)\ \big(=(-1,1)\big)$ in the sense of Gromov-Hausdorff, as $j\to\infty$. However, we remark that $\partial B_1(p_\infty)=\{-1,1\}$ and  that $\partial\big(X_\infty\backslash  \overline{B_1(p_\infty)}\big)=\varnothing.$

 Consider the first Dirichlet eigenvalue $\lambda_{1,j}$ of on $B_1(p_j)$. Because\ $\partial B_1(p_j)=\varnothing$, we have $Lip_0\big(B_1(p_j)\big)= Lip\big(B_1(p_j)\big)$. So the function $f=1$ is in $Lip_0\big(B_1(p_j)\big)$. This implies  $\lambda_{1,j}=0.$ On the other hand, it is obvious that the first Dirichlet eigenvalue of $B_1(p_\infty)$ is $\lambda_{1,\infty}=\pi^2/4$  (with eigenfunction $f(t)=\cos(\pi t/2)$).
\end{exam}
\begin{rem}\label{remark1.6}
 (1) Very recently, in an independent work \cite{aht17} by L. Ambrosio, S. Honda and D. Tewodrose, they show that the Weyl's law for eigenvalues $\lambda_j(X)$ of a whole compact $RCD^*(K,N)$-space (and the Neumann eigenvalues),  different from the local Dirichlet eigenvalues in this paper, holds if and only if
$$  \lim_{r\to 0}\int_X\frac{r^k}{\mu(B_r(x))}\du=\int_X  \lim_{r\to 0}\frac{r^k}{\mu(B_r(x))}\du,$$
where $k$ is the largest integer $k$ such that $\mu(\mathcal R_k) > 0$, and the $\mathcal R_k$ is the pieces in   the decomposition in \cite{mn14}. See also the constant $k_{\rm max}$ in Theorem~\ref{thm4.6}.

  (2) In another independent work \cite{ah17} by L. Ambrosio and S. Honda, they get that the same local spectral convergence result in Proposition \ref{prop1.3} holds if and only if the  following  condition holds:
$$W^{1,2}_0(B_R(p_\infty))= \cap_{\epsilon>0}W^{1,2}_0(B_{R+\epsilon}(p_\infty)).$$

(3) The condition $\mu\ll\mathscr H^N\ll\mu$ in Theorem \ref{thm1.1}  plays a role of non-collapsing.   G. De Philippis and  N. Gigli \cite{dg18} introduced the weak non-collapsed space by the condition $\nu\ll\mathscr H^N$, and very recently S. Honda \cite{honda19+} proved this implies $\mu=a\cdot\mathscr H^N$ for some constant $a\in(0,\infty)$ when $X$ is compact.\end{rem}

 \subsubsection*{Organization of the paper} In Section 2, we will provide some necessary
materials about $RCD^*(K,N)$ metric measure spaces and heat kernels on metric measure spaces.
In  Section 3, we will prove the locally uniformly convergence of heat kernels for a sequence of converging metric measure spaces. The main result Theorem~\ref{thm1.1} will be proved in Section 4.
At last, for the convenient of readers, we will give an appendix to introduce the dominated convergence theorem and Fatou's lemma for functions living on $pmGH$-converging metric measure spaces.

\subsubsection*{Acknowledgements} In the previous version we overlooked the condition  $\partial B_R(p_\infty)=\partial \big(X_\infty\backslash \overline{B_R(p_\infty)}\big)$ in Proposition \ref{prop1.3}. We appreciate Prof. S. Honda for showing us  Example \ref{example1.5}.
We are also grateful to Prof. Ambrosio and S. Honda for sharing us their  manuscripts \cite{ah17,aht17}.
We would like to  thank  the anonymous referees for very careful reading and  many useful suggestions.
We  thank also Prof. D. G. Chen,  B. B. Hua and Z. Q. Wang, and Dr. X. T. Huang for their interesting in the paper.
Both authors are partially supported by NSFC 11521101, and the first author is also partially supported
by NSFC 11571374.

\section{Preliminaries}

Let $(X,d)$ be a complete metric measure space  and $\mu$ be a Radon measure on $X$ with ${\rm supp}(\mu)=X.$  Given any $p\in X$ and $R>0$, we
denote by $B_R(p)$ the  ball centered at $p$ with radius $R$.

\subsection{Riemannian curvature-dimension conditions  \emph{RCD*}(\emph{K,N})}

Let $(X,d,\mu)$ be a metric measure space.
We denote by $\mathscr P_2(X,d)$  the $L^2$-Wasserstein space over $(X,d)$, i.e., the set of all Borel probability measures $\nu$ with
$$\int_Xd^2(x_0,x){\rm d}\nu(x)<\infty$$
for some (hence for all) $x_0\in X$. Given $\nu_1,\nu_2\in \mathscr P_2(X,d)$, the  $L^2$-Wasserstein distance between them  is defined by
$$ W_2^2(\nu_0,\nu_1):=\inf\int_{X\times X}d^2(x,y){\rm d}q(x,y)$$
where the infimum is taken over all couplings $q$ of $\nu_1$ and $\nu_2$, i.e., Borel probability measures $q$ on $X\times X$ with marginals $\nu_0$ and $\nu_1.$
The relative entropy is a functional on $\mathscr P_2(X,d)$, defined by
$$ {\rm Ent}(\nu):=\int_X\rho\ln \rho \du,$$
if $\nu=\rho\cdot\mu$ is absolutely continuous w.r.t. $\mu$ and $(\rho\ln \rho)_+$ is integrable. Otherwise we set ${\rm Ent}(\nu)=+\infty.$ Let $\mathscr P^*_2(X,d,\mu)\subset  \mathscr P_2(X,d)$ be the subset of all measures $\nu$ such that ${\rm Ent}(\nu)$ is finite.

 We set the function
\begin{equation*}
\sigma^{(t)}_k(\theta):=
\begin{cases}
\frac{\sin(\sqrt k\cdot t\theta)}{\sin(\sqrt k\cdot \theta)},& \quad 0<k\theta^2<\pi^2,\\
t, &\quad  k\theta^2=0,\\
\frac{\sinh(\sqrt{-k}\cdot t\theta)}{\sinh(\sqrt{- k}\cdot \theta)},& \quad k\theta^2<0,\\
\infty,&\quad  k\theta^2\gs\pi^2.
\end{cases}
\end{equation*}

\begin{defn}[\cite{eks15}]\label{definition2.1}
 Let $K\in\mathbb R$ and $N\in[1,\infty)$. A metric measure\ space $(X,d,\mu)$ is called to satisfy the \emph{entropy curvature-dimension condition} $CD^e(K,N)$ if and only if for each pair $\nu_0,\nu_1 \in \mathscr P^*_2(X,d,\mu)$ there exists a constant speed geodesic $(\nu_t)_{0\ls t\ls1}$ in $ \mathscr P^*_2(X,d,\mu)$ connecting $\nu_0$ to $\nu_1$ such that for all $t\in[0,1]$:
\begin{equation}\label{equa2.1}
\begin{split}
U_N(\nu_t)\gs  \sigma^{(1-t)}_{K/N}\big(W_2(\nu_0,\nu_1)\big)\cdot U_N(\nu_0)+ \sigma^{(t)}_{K/N}\big(W_2(\nu_0,\nu_1)\big)\cdot U_N(\nu_1),
\end{split}
\end{equation}
where $U_N(\nu):=\exp\big(-\frac{1}{N}{\rm Ent}(\nu)\big)$.
\end{defn}

Given a locally Lipschitz continuous function $f$ on $X$,  the \emph{pointwise Lipschitz constant} (\cite{che99}) of $f$ at $x$  is defined by
\begin{equation*}
{\rm Lip} f(x):=\limsup_{y\to x}\frac{|f(y)-f(x)|}{d(x,y)}=\limsup_{r\to0}\sup_{d(x,y)\ls r}\frac{|f(y)-f(x)|}{r},
\end{equation*}
and $ {\rm Lip}f(x)=0$ if $x$ is isolated. It is clear that ${\rm Lip}f$ is    $\mu$-measurable. The \emph{Cheeger energy}, denoted by ${\rm Ch}: L^2(X)\to[0,\infty]$, is defined \cite{ags14} by
$${\rm Ch}(f):=\inf\left\{\liminf_{j\to\infty}\frac 1 2\int_X {\rm Lip}^2 f_j{\rm d}\mu\right\},$$
where the infimum is taken over all sequences of Lipschitz functions $(f_j)_{j\in\mathbb N}$ converging to $f$ in $L^2(X)$. In general, ${\rm Ch}$ is a lower semi-continuous convex functional.
\begin{defn}[\cite{gig15}]\label{definition2.2}
 A metric measure space $(X,d,\mu)$ is called \emph{infinitesimally Hilbertian} if the associated Cheeger energy is quadratic.
\end{defn}

Several equivalent definitions for Riemannian curvature-dimension condition were introduced in \cite{ags-duke,eks15,ams16,gig13,gig15}. In this paper, we adapt the following notions    for the convenience.

\begin{defn}[\cite{eks15}]
 Let $K\in\mathbb R$ and $N\in[1,\infty)$. A metric measure\ space   $(X,d,\mu)$ is said to satisfy \emph{Riemannian curvature-dimension condition} $RCD^*(K,N)$,   if it is infinitesimally Hilbertian and satisfies the $CD^e(K,N)$ condition.
\end{defn}

Let $N>1$, the generalized Bishop-Gromov inequality for  $RCD^*(K,N)$ space (by a combination of  \cite[Corollary of 1.5]{grs16} and \cite[Remark 5.3]{stu06-2}) states that
 for any $p\in X$ and any $0<r<R$,
\begin{equation}\label{equa2.2}
\frac{\mu\big(B_R(p)\big)}{\int_0^R\mathfrak{s}^{N-1}_{\frac{K}{N-1}}(\tau){\rm d}\tau}\ls \frac{\mu\big(B_r(p)\big)}{\int_0^r\mathfrak{s}^{N-1}_{\frac{K}{N-1}}(\tau){\rm d}\tau},
\end{equation}
where the function $\mathfrak{s}_k(\tau)$ is given by
\begin{equation}\label{equa2.3}
\mathfrak{s}_{k}(\tau)=
\begin{cases}
\frac{\sin(\sqrt k\cdot \tau)}{\sqrt k}& {\rm if} \ \  k>0,\\
\tau & {\rm if} \ \ k=0,\\
\frac{\sinh(\sqrt {-k}\cdot \tau)}{\sqrt {-k}}& {\rm if} \ \  k<0.
\end{cases}
\end{equation}

Let $(X,d,\mu)$ be a metric measure space with  $RCD^*(K,N)$ for some $K\in\mathbb R$ and $N \gs 1$. We summarize some basic properties in \cite{stu06-2,ags-duke,agmr-trans,eks15} as follows:\\
$\bullet$ $(X,d)$ is a locally compact length space, i.e.,  for any $p,q\in X$, there is a shortest curve joined them;\\
$\bullet$ $(X,d,\mu)$  has a local measure doubling property on each ball $B_R(x)\subset X$. Moreover,  we have that, for all $0<r<R$,
\begin{equation}\label{equa2.4}
 \frac{\mu\big( B_R(p)\big)}{\mu\big( B_r(p)\big)}\ls \Big(\frac R r\Big)^N\cdot \exp\big(\sqrt{(N-1)|K\wedge0|}\cdot R\big):=C_{N,K,R}\cdot \Big(\frac R r\Big)^N;
 \end{equation}
$\bullet$   $(X,d,\mu)$ supports a local $L^2$-Poincar\'e inequality on each ball $B_R(x)\subset X$. Moreover, the Poincar\'e constant  $C_P(N,K,R)$ depends only on  $N$ and $\sqrt{|K\wedge0|}R$.\\
$\bullet$   The canonical Dirichlet form $({\rm Ch}, D({\rm Ch}))$ is strongly local and regular, and admits a Carr\'e du champ $\Gamma(f) $ for each $f\in D({\rm Ch})$. Moreover, the intrinsic distance $d_{\rm Ch}$ induced by  $({\rm Ch}, D({\rm Ch}))$ coincides with the original distance $d$ on $X$ (see \cite{ags-duke,agmr-trans});\\
$\bullet$ The heat kernel $H(x,y,t)$ on $X$ exists  (see \cite{stu95} and \cite[Theorem 1.2]{jlz16}), and   there is a positive constant $C_{N,K}$,  depending only on $N$ and $K\wedge 0$, such that
\begin{equation}\label{equa2.5}
H(x,y,t)\ls \frac{C_{N,K}}{\mu\big(B_{\sqrt t}(x)\big) }\exp\left(-\frac{d^2(x,y)}{5t}+C_{N,K}\cdot t\right).
\end{equation}

\subsection{Sobolev spaces, local Dirichlet heat kernels and Dirichlet eigenvalues}

 Several different notions of Sobolev spaces for   metric measure spaces  have been given in  \cite{che99,shan00,ags14,ags13-lip,haj96,hk00}.  In this paper, we will pay our attentions to the $RCD^*(K,N)$-spaces for some $K\in\mathbb R$ and $N\gs 1$.
  In the case, the notions of Sobolev spaces in \cite{che99,shan00,ags14,ags13-lip}  coincide each other (see, for example, \cite{ags14,ags13-lip}), and they have the equivalent norms with the notion  of Sobolev spaces in \cite{haj96,hk00}.

Let $(X,d,\mu)$ be a metric measure space with  $RCD^*(K,N)$ for some $K\in\mathbb R$ and $N \geq 1$. For an open subset $\Omega\subset X$, we denote by    $Lip_{\rm loc}(\Omega)$ (and $Lip_0(\Omega)$), the set of all locally Lipschitz continuous functions on $\Omega$ (and the set of all locally Lipschitz continuous functions $f$ such that $d({\rm supp}(f),\partial\Omega)>0$, respectively).
Let $p\in[1,\infty]$ and let $f\in L^p(\Omega)\cap Lip_{\rm loc}(\Omega)$. The $W^{1,p}(\Omega)$-norm, $\|f\|_{1,p}$, is given by
$$\|f\|_{1,p}:=\|f\|_p+ \|{\rm Lip}f\|_p,$$
here and in the sequel, we denote $\|f\|_p:=\|f\|_{L^p}.$
 The Sobolev space $W^{1,p}(\Omega)$ is defined to be  the completion of  all   locally Lipschitz continuous, $f$, for which $\|f\|_{1,p}<\infty$, with respect to
 the norm $\|f\|_{1,p}$.  Given $p\in(1,\infty)$, it was proved \cite{che99,ags13-lip}
, for each $f\in W^{1,p}(\Omega)$, that there exists a function $|\nabla f|\in L^p(\Omega)$, called the  \emph{minimal weak upper gradient}, such that
$$\|f\|_{1,p}=\|f\|_p+\||\nabla f|\|_p.$$
For a locally Lipschitz function $f\in W^{1,p}(\Omega)$, it was showed \cite{che99} that $|\nabla f|={\rm Lip}f$ a.e. in $\Omega$.
We say that a function $f\in W^{1,p}_{\rm loc}(\Omega)$ if $f\in W^{1,p}(\Omega')$ for every open subset $\Omega'\subset\subset\Omega.$  We refer the readers to \cite{che99,shan00,ags13-lip,gig15} for further information of these Sobolev spaces.

For $1<p<\infty$, let us recall from \cite{hm96} that the \emph{Sobolev p-capacity} of the set  $E\subset X$:
\begin{align*}
	{\rm Cap}_p(E):=\inf\big\{&\|f\|^p_{W^{1,p}(X)}:\ f\in W^{1,p}(X)\\
       &{\rm such \ that }\ f\gs 1\ {\rm on\ a\ neighborhood\ of }\ E\big\}.\text{\footnotemark}
\end{align*}
\footnotetext{In \cite{hm96}, the definition of Sobolev $p$-capacity was given via the Sobolev norm in \cite{haj96}. Meanwhile, according to \cite{shan00}, the Sobolev norms in \cite{haj96} is equivalent  to the one in \cite{che99,shan00,ags13-lip}. Therefore, the following both definitions of $p-$quasi everywhere and $p$-quasi continuity concide with the corresponding definitions in \cite{hm96}.}\noindent If there is no such a function $f$, we set ${\rm Cap}_p(E)=\infty.$ It is clear that ${\rm Cap}_p(\overline E)={\rm Cap}_p(E).$
  An equivalent definition is given in \cite{shan00}, see for instance \cite[Theorem 3.4]{kkm00} and \cite{shan01}.

A property holds $p$-q.e. ($p$-quasi everywhere), if it holds except of a set $Z$ with ${\rm Cap}_p(Z)=0$. Since ${\rm Cap}_p(\overline{Z})={\rm Cap}_p(Z),$ we may assume that the except set $Z$ is closed. A function $f \!:\!X\!\to\![-\infty,\infty]$ is called $p$-$quasi\ continuous$ in $X$ if for each $\epsilon>0$, there is a set $F_\epsilon$ such that ${\rm Cap}_p(F_\epsilon)<\epsilon$ and the restriction $f|_{X\backslash F_\epsilon}$ is continuous. We may also assume that $F_\epsilon$ is closed.

It is well-known that any $W^{1,p}$-function $f$ has a $p$-quasi continuous representative (see \cite{hm96}). We will always use such a representative in this paper.
In \cite[Theorem 3.2]{kkm00}, it is proved that, for any two $p$-quasi continuous functions $f$ and $g$, if $f=g$ $\mu$-a.e. in an open set $O$, then  $f=g$ $p$-q.e. in $O$.

\begin{defn}[\cite{kkm00}]\label{definition2.4}
Let $1<p<\infty$ and  $E\subset X$,  a function $f$ on $E$ is called to belong to the Sobolev space with zero boundary values, denoted by $f\in W^{1,p}_0(E)$, if there exists a $p$-quasi continuous function $\tilde f\in W^{1,p}(X)$ such that $\tilde f=f$ $\mu$-a.e. in $E$ and $\tilde f=0$ $p$-q.e. in $X\backslash E$.
\end{defn}

 According to \cite[Remark 5.10]{kkm00} (see also \cite[Theorem 4.8]{shan01}),  the space $W_0^{1,p}(\Omega)=H^{1,p}_0(\Omega)$, which is the closure of $Lip_0(\Omega)$ under the $W^{1,p}(\Omega)$-norm.
Given any open set $\Omega\subset X$, it is clear that  $W^{1,p}_0(\Omega)\subset W^{1,p}_0(\overline{\Omega})$. However, generally speaking,      $W^{1,p}_0(\Omega)\neq W^{1,p}_0(\overline{\Omega})$.
\begin{lem}\label{lemma2.5}
Let $O\subset X$ be an open set and let $1<p<\infty$. Suppose that $f $ is  a $p$-quasi continuous in $X$ and that $f=0$ $p$-q.e. in $O$. Then we have
that $f=0$ $p$-q.e. in $\overline O.$
\end{lem}
\begin{proof}
From the definition, we know that  $f=0$ $p$-q.e. in $O.$ So it suffices to show that   $f=0$ $p$-q.e. in $\partial O.$ We can assume    ${\rm Cap}_p(\partial O)>0$. Otherwise, it is nothing to do.

We will argue by a contradiction. Suppose that there is a subset $A\subset \partial O$ such that ${\rm Cap}_p(A)>0$ and that $f(x)\not=0$ for any $x\in A$.

Taken arbitrarily  $\epsilon\in (0,{\rm Cap}_pA/2),$ since $f $ is  a $p$-quasi continuous in $X$, we can find a closed set $F_\epsilon$ with ${\rm Cap}_p(F_\epsilon)<\epsilon$ and the restriction $f|_{X\backslash F_\epsilon}$ is continuous. Noting that  $f=0$ $p$-q.e. in $O,$ i.e., there exists a closed set $Z$ with ${\rm Cap}_p(Z)=0$ such that $f=0$ on $O\backslash Z$. We have that the restriction $f|_{O\backslash (F_\epsilon\cup Z)}\equiv0$ and that $f|_{X\backslash (F_\epsilon\cup Z)}$ is continuous.

By ${\rm Cap}_p(F_\epsilon\cup Z)<\epsilon< {\rm Cap}_pA/2$,  we can find a point $x_0\in A\backslash (F_\epsilon \cup Z).$  There exists a sequence $\{x_j\}_{j=1}^\infty\subset O$ with $\lim_{j\to\infty}x_j=x_0$, since $x_0\in \partial O$. Noting that $F_\epsilon\cup Z$ is closed and $x_0\not\in F_\epsilon\cup Z$, we know  that $x_j \not\in F_\epsilon\cup Z $ for all sufficiently large $j$. By combining the facts  that $f|_{X\backslash (F_\epsilon\cup Z)}$ is continuous at $x_0$ and that $f(x_j)=0$ for all large $j$ (since $x_j\in O\backslash (F_\epsilon\cup Z)$ for all large $j$), we conclude that   $f(x_0)=0.$ This contradicts with $x_0\in A$, and hence we finish the proof.
\end{proof}

\begin{cor}\label{corollary2.6}
Let $\Omega\subset X$ be an open set and let $1<p<\infty$. If $\partial \Omega=\partial(X\backslash\overline\Omega)$, then  $W^{1,p}_0(\Omega)= W^{1,p}_0(\overline{\Omega})$.
\end{cor}
\begin{proof}
It suffices to show $W^{1,p}_0(\Omega)\supset W^{1,p}_0(\overline{\Omega})$. Given any $f\in W^{1,p}_0(\overline{\Omega})$, there exists   a $p$-quasi continuous function $\tilde f$  in $X$ such that $\tilde f=f$ $\mu$-a.e. in $\overline{\Omega}$ and that $\tilde f=0$ $p$-q.e. in $X\backslash\overline{\Omega}.$ By applying Lemma \ref{lemma2.5} to $\tilde f$ and $O:=X\backslash\overline{\Omega}$, we conclude that   $\tilde f=0$ $p$-q.e. in $\overline {X\backslash\overline{\Omega}}.$ The assumption $\partial \Omega=\partial(X\backslash\overline\Omega)$ implies
$$\overline {X\backslash\overline{\Omega}}=(X\backslash\overline\Omega)\cup \partial(X\backslash\overline\Omega)=(X\backslash\overline\Omega)\cup \partial\Omega   =X\backslash \Omega.$$
Therefore, we get that  $\tilde f=0$ $p$-q.e. in $X\backslash \Omega.$ Noting that $\tilde f=f$ $\mu$-a.e. in $\Omega\subset\overline{\Omega}$, we have $f\in W^{1,2}_0(\Omega)$, by  Definition \ref{definition2.4}.
The proof is finished.
\end{proof}
\begin{rem}\label{rem2.4}
(1) In fact, in   Corollary \ref{corollary2.6}, we only need to assume that
$${\rm Cap}_p\big(\partial \Omega\backslash\partial(X\backslash\overline\Omega)\big)=0.$$
(2) The space $W^{1,p}_0(\overline{\Omega})$ is equivalent to the space $\hat{H}^{1,p}_0(\Omega)$ given in  \cite{ah17} by Ambrosio-Honda.
\end{rem}

Let $(X,d,\mu)$   be an $RCD^*(K,N)$ metric measure space with some $K\in \mathbb R$ and some $N\gs1$.   Given any  bounded open set $\Omega\subset X$ and $p\in(1,\infty)$, according to \cite[\S 4.3]{gig15}, the space $W^{1,2}(\Omega)$ is  a Hilbert space, and  for any $f,g\in W^{1,2}_{\rm loc}(\Omega)$, the inner product $\ip{\nabla f}{\nabla g}$ is well defined in $ L^1_{\rm loc}(\Omega)$.
In the sequel of the paper, we will always denote that
$$H^1(\Omega):=W^{1,2}(\Omega), \quad H^1_0(\Omega):=W_0^{1,2}(\Omega) \quad{\rm and}\quad H^1_{\rm loc}(\Omega):=W^{1,2}_{\rm loc}(\Omega).$$

 For any fixed bounded domain $\Omega\subset  X$,  we consider the canonical Dirichlet form $(\mathscr E_{\Omega},H^1_0(\Omega))$, where
\begin{equation}\label{equa2.6}
\mathscr E_\Omega(f):=\int_\Omega|\nabla f|^2\du,\quad f\in H^1_0(\Omega).
\end{equation}
This canonical Dirichlet form is strongly local and regular   (see, for example, the proof of \cite[Lemma 6.7]{ags-duke}). Indeed, the strong locality is  a consequence of the locality of minimal weak upper gradients and the regularity comes from the density of Lipschitz functions in $H^{1}_0(\Omega).$
  The associated infinitesimal generator of $(\mathscr E_{\Omega},H^1_0(\Omega))$, denoted by $\Delta_\Omega$ with domain $D(\Delta_\Omega)$, is a non-positive definite self-adjoint operator, and the associated analytic semi-group is $(H_tf)_{t\gs0}$ for any $f\in L^2(\Omega)$.
 If ${\rm diam}(\Omega)\ls {\rm diam}(X)/s$ for some $s >1$,   a compact embedding of $H^{1}_0(\Omega)$ into $L^2(\Omega)$ was proved in \cite{hk00} (see also \cite{gms15} for $RCD^*(K,\infty)$-spaces for some $K\in \mathbb R$, or \cite[Eq.(5.2)]{bm06} for the spaces with a local measure doubling property and  a local $L^2$-Pincar\'e inequality, by the equivalence of the Sobolev norms in \cite{che99, ags14,ags13-lip} and in \cite{haj96,hk00}). Hence  the operator $(Id-\Delta_\Omega)^{-1}$ is compact. The spectral theorem implies that spectrum is discrete (see, for example \cite{dav89}). We denote by
$$0<\lambda_1^{\Omega}\ls \lambda_2^{\Omega}\ls \cdots\ls \lambda_m^{\Omega}\ls \cdots, \quad j\in \mathbb N,$$
the (Dirichlet) eigenvalues of  $\Delta_\Omega$. For each $\lambda_m^{\Omega}$, the associated eigenfunction is $\phi_m^{\Omega}$, i.e.,
\begin{equation}\label{equa2.7}
\Delta_\Omega \phi_m^{\Omega}=- \lambda_1^{\Omega}\phi_m^{\Omega}.
\end{equation}
 We normalize them so that $\|\phi_m^{\Omega}\|_2=1 $ for each $m\in\mathbb N$. It is well-known that the sequence $\{\phi_m\}_{m\in\mathbb N}$ forms a complete basis of $L^2(\Omega)$, and that the (local) Dirichlet heat flow is given by
 $$ H_tf(x)=\int_\Omega H^\Omega(t,x,y)f(y)\du,\quad t\gs0,\ \ \forall f\in L^2(\Omega),$$
  where
\begin{equation}\label{equa2.8}
H^{\Omega}(t,x,y)=\sum_{m\gs1}e^{-\lambda^{\Omega}_m}\phi^{\Omega}_m(x)\phi^{\Omega}_m(y),\quad \forall (x,y,t)\in \Omega\times\Omega\times(0,\infty).
\end{equation}
 is the (local) Dirichlet heat kernel  (the fundamental solution of the heat equation with Dirichlet boundary value).

The weak maximum principle implies the monotonicity of Dirichlet heat kernels with respect to domains. Namely, given two domains $\Omega\subset \Omega'\subset X$, we have
$$H^{\Omega}(t,x,y)\ls H^{\Omega'}(t,x,y) ,\quad \forall (x,y,t)\in \Omega\times\Omega\times(0,\infty).$$
The existence and Gaussian bounds of the global heat kernels have been established in \cite{stu95} on $(X,d,\mu)$. Thus for a sequence of balls $\{B_{R_j}(x_0)\}$ with $R_j\nearrow\infty$, the heat kernels $H^{B_{R_j}(x_0)}(x,y,t) $ converge  to a global heat kernel $H(x,y,t)$ on $X\times X\times(0,\infty)$, as $R_j\nearrow \infty$.

 Let us recall the definition of the  \emph{distributional Laplacian}.  Given a function $f\in H^1_{\rm loc}(\Omega)$,  the distributional Laplacian $\mathscr Lf$ is defined as a functional
\begin{equation}\label{equa2.9}
\mathscr Lf(\phi):=-\int_\Omega\ip{\nabla f}{\nabla \phi}\du,\quad \forall \phi\in H^1_0(\Omega)\cap L^\infty(\Omega).
\end{equation}
If $f\in H^1(\Omega)$, then $\mathscr Lf$ can be extended to a functional on $H^1_0(\Omega).$ It is clear that if $f\in D(\Delta_{\Omega})$ and $\Delta_\Omega f=g,$ then $\mathscr Lf= g\cdot\mu $ in the sense of distributions. Conversely,
it was proved \cite{gig15} that any $f\in H^1_0(\Omega)$, if there is   $g\in L^2(\Omega)$ such that $\mathscr Lf= g\cdot\mu $ in the sense of distributions, then $f\in D(\Delta_{\Omega})$ and $\Delta_\Omega f=g.$

\subsection{Pointed measured Gromov-Hausdorff convergence}

A pointed metric measure space $(X,d,\mu,p)$ is a metric measure space\ $(X,d,\mu)$ with a base point $p\in {\rm supp}(\mu)$. Recall that we always assume\ ${\rm supp}(\mu)=X$.
\begin{defn}\label{definition2.8}
Let $(X_j,d_j,\mu_j,p_j)$,   $j\in\mathbb N\cup\{\infty\},$ be a sequence of pointed metric measure spaces. It is said that  $(X_j,d_j,\mu_j,p_j)_{j\in\mathbb N}$ converge  to  $(X_\infty,d_\infty,\ \mu_\infty,p_\infty)$, as $j\to\infty$, in the sense of \emph{pointed measured Gromov-Hausdorff} topology, denoted by
  $$ (X_j,d_j,\mu_j,p_j)\overset{pmGH}{\longrightarrow} (X_\infty,d_\infty,\mu_\infty,p_\infty),$$
 if for any fixed $\epsilon,R>0$, there exists a constant $N(\epsilon,R)>0$ such that, for every $j\gs N(\epsilon,R)$, there exists a Borel map  $\Phi_j^{\epsilon,R}:B_R(p_j)\to X_\infty$ such that
\begin{itemize}
	\item[(1)] $\Phi_j^{\epsilon,R}(p_j)=p_\infty$;
	\item[(2)]  for all $x,y\in B_R(p_j)$, $\ |d_\infty\big(\Phi_j^{\epsilon,R}(x),\Phi_j^{\epsilon,R}(y)\big)- d_j(x,y)|\ls \epsilon;$
	\item[(3)] the $\epsilon$-neighborhood of $\Phi_j^{\epsilon,R}(B_R(p_j))$ contains $B_{R-\epsilon}(p_\infty)$;
	\item[(4)] the Levi metric $\rho_L$ between the measures $(\Phi_j^{\epsilon,R})_\sharp(\mu_j|_{B_R(p_j)})$ and\ $\mu_\infty|_{B_R(p_\infty)}$ is less  than $\epsilon$, for almost all $R>0$. Here the Levi metric $\rho_L(\nu_1,\nu_2)<\epsilon$ for two measures $\nu_1,\nu_2$ if and only if for any $\delta>0$, the $\delta$-neighborhood $A_\delta$ of $A$, there hold
 $$\nu_1(A)\ls \nu_2(A_\delta)+\epsilon\quad {\rm and}\quad \nu_2(A)\ls \nu_1(A_\delta)+\epsilon.$$
\end{itemize}
Such maps $\Phi^{\epsilon,R}_j$ are called   $\epsilon$-mGH approximations. Remark that the Levi metric convergence is equivalent to the measure's weak convergence.
 \end{defn}

Recall that any $RCD^*(K,N)$-space is a length space.  The pointed measured Gromov-Hausdorff convergence  on length spaces can be given as follows (see, for example, \cite[Remark 3.29]{gms15}).
\begin{prop}\label{prop2.9}
Let $(X_j,d_j,\mu_j,p_j)$,   $j\in\mathbb N\cup\{\infty\},$ be a sequence of pointed metric measure spaces. Assume that  all $(X_j,d_j)_{j\in\mathbb N}$ are length spaces. Then
  $$ (X_j,d_j,\mu_j,p_j)\overset{pmGH}{\longrightarrow} (X_\infty,d_\infty,\mu_\infty,p_\infty)$$
  is equivalent to the following:$\ $
 There exist sequences $R_j\nearrow\infty$, $\epsilon_j\searrow0$ and Borel maps  $\Phi_j:X_j\to X_\infty$ such that
\begin{itemize}
	\item[(1')] $\Phi_j(p_j)=p_\infty$;
	\item[(2')]  for all $x,y\in B_{R_j}(p_j)$, $\ |d_\infty\big(\Phi_j(x),\Phi_j(y)\big)- d_j(x,y)|\ls \epsilon_j$ and\\ $\Phi_j(B_{R_j}(p_j))\subset B_{R_j}(p_\infty)$;
	\item[(3')] the $\epsilon_j$-neighborhood of $\Phi_j(B_{R_j}(p_j))$ contains $B_{R_j}(p_\infty)$;
	\item[(4')] the measures $(\Phi_j)_\sharp(\mu_j)$ weakly converges to $\mu_\infty$ as $j\to\infty$, that is, for any $\phi\in C_0(X_\infty)$,
 $$\lim_{j\to\infty}\int_{X_j}\phi\circ \Phi_j{\rm d}\mu_j= \int_{X_\infty}\phi{\rm d}\mu_\infty.$$
\end{itemize}
 \end{prop}
Given a sequence of points $\{x_j \in X_j\}_{j\in\mathbb N\cup\{\infty\}}$, we say that $x_j\to x_\infty$  with respect to the sequences $(\epsilon_j)$ and maps $(\Phi_j)$  if and only if $d_\infty\big(x_\infty, \Phi_j(x_j)\big)<\epsilon_j$ for all $j\in\mathbb N$.  Here both $(\epsilon_j)$ and  $(\Phi_j)$ are given in the   Proposition \ref{prop2.9}.   Below, we will sometimes write $x_j\to x_\infty$ without mention of the particular choices of  $(\epsilon_j)$ and  $(\Phi_j)$.

We refer the readers to \cite{gms15} for some other notions of convergence for pointed metric measure spaces. We also consider the convergence of functions on a sequence of converging pointed metric measure spaces.
\begin{defn}\label{definition2.10}
 Let  $(X_j,d_j,\mu_j,p_j)_{j\in\mathbb N\cup\{\infty\}}$ be a sequence of   pointed metric measure spaces. Assume that all $(X_j,d_j)$ are length spaces and that
 $$(X_j,d_j,\mu_j,p_j)\overset{pmGH}{\longrightarrow} (X_\infty,d_\infty,\mu_\infty,p_\infty) $$ with the sequences $(\epsilon_j)$ with $\epsilon_j \searrow0$ and   maps $(\Phi_j)$ as in Proposition \ref{prop2.9}.
 Let $R>0$. Suppose that $\{f_j\}_{j\in\mathbb N\cup\{\infty\}}$ is a sequence of Borel functions on $B_R(p_j)$.
It is said that:

\smallskip
\noindent
$(i)$ $f_j\rightarrow f_\infty$ over $B_R(p_j)$ at point $x_\infty\in B_R(p_\infty)$, if $f_j(x_j)\to f_\infty(x_\infty)$ for any sequence $x_j\in X_j$ such that $\Phi_j(x_j)\to x_\infty$ in $X_\infty$. Precisely,    for any $\varepsilon>0$, there exist $N(\varepsilon,x_\infty)\in\mathbb N$  and $\delta(\varepsilon,x_\infty)>0$ such that
$$\sup_{x\in B_R(p_j),\ d_\infty(\Phi_j(x),x_\infty)<\delta(\varepsilon,x_\infty)}|f_j(x)-f_\infty(x_\infty)|<\varepsilon,\quad \forall\ j\gs N(\varepsilon,x_\infty);$$
$(ii)$
 $f_j\rightarrow f_\infty$ \emph{uniformly} over $B_R(p_j)$, if for any $\varepsilon>0$ there exist $N(\varepsilon)\in\mathbb N$ and $\delta(\varepsilon)>0$ such that
$$\sup_{x\in B_R(p_j),\ y\in B_R(p_\infty),\ d_\infty(\Phi_j(x),y)<\delta(\varepsilon)}|f_j(x)-f_\infty(y)|<\varepsilon,\quad \forall\ j\gs N(\varepsilon).$$
\end{defn}
\begin{rem}\label{remark2.11} (1)
The pointwise and uniform convergence of functions defined on varying space have been given in \cite{mn14}  via an extrinsic point of view.
This definition $(i)$ is equivalent to the pointwise convergence in Definition 2.11 in \cite{mn14}. If the limit function $f_\infty$ is uniformly continuous on $B_R(p_\infty)$, then this definition $(ii)$ is equivalent to the uniform convergence in  Definition 2.11 in \cite{mn14}.

(2) It is well know (see, for example, \cite[\S 3]{cheeger-d}) that if $f_j\to f_\infty$ over $B_R(p_\infty)$ then $f_\infty$ is continuous. Indeed, it can be seen as follows. Suppose not, there exist $\{y_\alpha\}_{\alpha\in\mathbb N\cup\{\infty\}}\subset B_R(p_\infty)$ such that  $y_\alpha\to y_\infty$ as $\alpha\to\infty$ and $|f_\infty(y_\alpha)-f_\infty(y_\infty)|\gs \varepsilon_0$ for some $\varepsilon_0>0$. Fixed each $\alpha\in\mathbb N$, we can find a sequence $y_{j,\alpha}\in X_j$ such that $\Phi_j(y_{j,\alpha})$ converge to $y_\alpha$  and $f_j(y_{j,\alpha})\to f_\infty(y_\alpha)$ as $j\to\infty$.  Now for sufficiently large $j_\alpha$, $j_\alpha\gs N(\alpha,y_\alpha,\varepsilon_0),$ we have $|f_{j_\alpha}(y_{j_\alpha})-f_\infty(y_\alpha)|<\varepsilon_0/3$. By a diagonal  argument, there exists a subsequence $y_{j_\alpha,\alpha}$ converging to $y_\infty $ as $\alpha\to\infty$. Hence, we get $|f_{j_\alpha}(y_{j_\alpha})-f_\infty(y_\infty)|<\varepsilon_0/3$ for large enough $\alpha.$
This is a contradiction.
\end{rem}
We remark that the Arzela-Ascoli theorem can be generalized to the case where the functions live on different spaces (see, for example, \cite{lv09} or Proposition 2.12 in \cite{mn14}). We also need the following lemma:
\begin{lem}[{\cite[Lemma 10.7]{che99}}]\label{lemma2.12}
Let $R>0$ and let $(X_j,d_j,\mu_j,p_j)_{j=1,2}$ be two  pointed metric measure spaces with $RCD^*(K,N)$ for some $K\in\mathbb R$ and  $N\gs1$.
Assume that $d_{mGH}\big(B_R(p_1),B_R(p_2)\big)<\varepsilon$, for some  $\varepsilon>0$, with an $\varepsilon$-mGH-approximation $\Phi=\Phi^{\varepsilon,R}:B_R(p_1)\to B_R(p_2)$  (see   Definition \ref{definition2.8}).

  If  $f_1$ is a Lipschitz function on $B_R(p_1)$ with $\||\nabla f_1|\|_{L^\infty(B_R(p_1))}\ls L$, then there exists a  Lipschitz function $f_2$ on $B_R(p_2)$ such that
  \begin{align*}
	  \|f_2\circ\Phi -f_1 \|_{L^\infty(B_R(p_1))} &\ls \kappa(\varepsilon),\\[1ex]
	  \||\nabla f_2|\|_{L^\infty(B_R(p_2))}&\ls\big( L+\kappa(\varepsilon)\big),\\[1ex]
	  \int_{B_R(p_2)}|\nabla f_2|^2\du_2&\ls \int_{B_R(p_1)}|\nabla f_1|^2\du_1+\kappa(\varepsilon),
\end{align*}
where $\kappa(\varepsilon):=\kappa_{N,K,R,L}(\varepsilon)$ is a positive function, depending on  $N,K,R$ and $L$,with $\lim_{\varepsilon\to0}\kappa(\varepsilon)=0$.
\end{lem}

The lower semi-continuity of  Dirichlet energy on   converging  spaces is given in \cite[Proposition 2.13]{mn14} and \cite[Theorem III]{gms15}. The following special case is enough for our purpose in this paper.

\begin{lem}[Lower semi-continuity of the energy]\label{lemma2.13}
Let $R>0$.
Let $(X_j,d_j,\mu_j,p_j)_{j\in\mathbb N\cup\{\infty\}}$  be a sequence of   pointed metric measure spaces. Assume that all $(X_j,d_j,\mu_j)$   satisfy $RCD^*(K,N)$  for some $K\in \mathbb R$ and $N\gs1$ and that $$ (X_j,d_j,\mu_j,p_j)\overset{pmGH}{\longrightarrow} (X_\infty,d_\infty,\mu_\infty,p_\infty).$$

If  $\{f_j\}_{j\in\mathbb N\cap\{\infty\}}$ is a sequence of  Lipschitz functions on $B_R(p_j)$, for each $j\in\mathbb N\cup\{\infty\}$, and $f_j \rightarrow f_\infty$ uniformly over $B_R(p_j)$, and if there exists $C_1$ such that
\begin{equation}\label{equa2.10}
\sup_{j\in \mathbb N}\|\nabla f_j\|_{L^\infty(B_R(p_j))}\ls C_1,
\end{equation}
then we have
$$ \liminf_{j\to\infty} \int_{B_R(p_j)}|\nabla f_j|^2\du_j\gs \int_{B_R(p_\infty)}|\nabla f_\infty|^2\du_\infty.$$
\end{lem}
\begin{proof}
For completeness, we sketch a proof.

By using Lemma \ref{lemma2.12} to each $f_j$, we can find a Lipschitz function $g_j$ on $B_R(p_\infty)$ such that
 \begin{equation*}
  \|g_j\circ\Phi_j -f_j \|_{L^\infty(B_R(p_j))}\ls \kappa(\epsilon_j),
  \end{equation*}
and
\begin{equation*}
  \int_{B_R(p_\infty)}|\nabla g_j|^2\du_\infty\ls \int_{B_R(p_j)}|\nabla f_j|^2\du_j+\kappa(\epsilon_j),
\end{equation*}
where $(\Phi_j),(\epsilon_j)$ are given in Proposition \ref{prop2.9} and $\kappa(\epsilon_j):=\kappa_{N,K,R,C_1}(\epsilon_j)\to 0$ as $\epsilon_j\to0$, and where the constant $C_1$ is in (\ref{equa2.10}).  Then, we get that
$g_j\to f_\infty$ in $L^\infty(B_R(p_\infty))$  and that
\begin{equation*}
\liminf_{j\to\infty} \int_{B_R(p_j)}|\nabla f_j|^2\du_j\gs \liminf_{j\to\infty} \int_{B_R(p_\infty)}|\nabla g_j|^2\du_\infty.
\end{equation*}
Now the assertion follows, by  the lower semi-continuity of energy on a fixed space, see \cite[Theorem 2.5]{che99}.
\end{proof}

\section{The converge of Dirichlet heat kernels}

 In this section, we will discuss the convergence of  the local Dirichlet heat kernels on different  pointed metric measure spaces.

\subsection{Convergence of functions living on pmGH-converging spaces}

We fix a sequence of  pointed metric measure spaces $(X_j,d_j,\mu_j,p_j)_{j\in\mathbb N\cup\{\infty\}}$ such that
$$ (X_j,d_j,\mu_j,p_j)\overset{pmGH}{\longrightarrow} (X_\infty,d_\infty,\mu_\infty,p_\infty).$$
Throughout of this subsection,  we always assume that, for each $j\in\mathbb N$, $(X_j,d_j,\mu_j)$  satisfies $RCD^*(K,N)$ for some $K\in\mathbb R$ and   $N\gs1$. Then the limit space $(X_\infty,d_\infty,\mu_\infty)$ does so, by the stability of the $RCD^*$-condition under $pmGH$-convergence.

Let us first introduce the notions of $L^2$-convergence and  $H^1$-convergence for functions living on varying spaces $X_j$. We will adapt an intrinsic point of view for the definitions, similar as in \cite{ding02,ks03,hon11} .
We refer also readers to \cite{gms15,ah16} for some similar concepts of convergence via an extrinsic point of view.
\begin{defn} \label{def3.1}
Let $R>0$.
\begin{itemize}
\item[(1)]  Suppose that $\{f_j\}_j\in L^2(B_R(p_j))$  for each $j\in\mathbb N\cup\{\infty\}$.  We say that $f_j\to f_\infty$ in $L^2(B_R(p_j))$ if  we have   $f_j\to f_\infty$  over $B_R(p_j)$  $\mu_\infty$-a.e. (in the sense that $f_j\to \hat f$ for some $\hat f$ with $f_\infty(x)=\hat f(x)$ $\mu_\infty$-a.e. $x\in B_R(p_\infty)$), and if
\begin{equation*}
\lim_{j\to \infty}\int_{ B_R(p_j)}|f_j|^2\du_j=\int_{ B_R(p_\infty)}|f_\infty|^2\du_\infty.
\end{equation*}
\item[(2)]  Suppose that $\{f_j\}_j\in H^1(B_R(p_j))\ \big(:=W^{1,2}(B_R(p_j))\big)$  for each $j\in\mathbb N\cup\{\infty\}$. We say that $f_j\to f_\infty$ in $H^{1}(B_R(p_j))$ if  it holds $f_j\to f_\infty$ in $L^2(B_R(p_j))$  and
\begin{equation*}
\lim_{j\to\infty}\int_{B_R(p_j)}|\nabla f_j|^2\du_j=\int_{B_R(p_\infty)}|\nabla f_\infty|^2\du_\infty.
\end{equation*}
\end{itemize}
 \end{defn}

It is not hard to see  that  if $f_j\to f_\infty$ in $L^2(B_R(p_j))$ in the above Definition~\ref{def3.1} (i), then their zero extensions $\tilde{f}_j$ (that is, $\tilde{f}_j=f_j$ in $B_R(p_j)$ and $\tilde{f}_j=0$ in $X_j\backslash B_R(p_j)$) strongly $L^2$-converge to $\tilde{f}_\infty$ in the sense of \cite{gms15} (see also \cite{ah16}). Indeed, by using  the weak compactness of $\{\tilde{f}_j\}$ in $L^2(X_j)$  (see, page 1115 on \cite{gms15}),  we get that $\tilde{f}_j$ weakly $L^2$-converge to $\tilde{f}_\infty$ in the sense of \cite{gms15}.  From the Definition~\ref{def3.1} (i), we have also that $\|\tilde{f}_j\|_{L^2(X_j)}\to \|\tilde{f}_\infty\|_{L^2(X_\infty)}.$

Similar as in the case of functions on a fixed space, it is available that the dominated convergence theorem for functions living on $pmGH$-converging spaces. In particular, if $\{f_j\}_{j\in\mathbb N\cup\{\infty\}}$ is a sequence  of functions such that $f_j\to f_\infty$  over $B_R(p_j)$ at $\mu_\infty$-almost all points in $B_R(p_\infty)$ and that they are bounded uniformly, then  $f_j\to f_\infty$ in $L^2(B_R(p_j))$. For convenient, we will give some detailed information, in the Appendix A, for the dominated convergence theorem and Fatou's lemma for functions living on varying spaces.

Let us sum up some basis properties on these convergence.

\begin{prop}\label{prop3.2}
  Let  $R>0$.
\begin{itemize}
\item[(i)] Assume that  $\partial B_R(p_\infty)=  \partial\big(X_\infty\backslash \overline{B_R(p_\infty)}\big)$. If $g_j\in H^1_0(B_R(p_j))$ with $\||\nabla g_j|\|_2\ls C$ for some $C>0$, for all $j\in\mathbb N$, and if $g_j\to g_\infty$ in\ $L^2(B_R(p_j))$, then we have $g_\infty\in  H^1_0(B_R(p_\infty)).$
\item[(ii)]  Let $\{f_j\}_{j\in\mathbb N\cap\{\infty\}}$ be a  sequence of Lipschitz functions on $B_R(p_j)$ such that $f_j\to f_\infty$ uniformly over $B_R(p_j)$. Suppose that
 \begin{equation}\label{equa3.1}
 \sup_{j\in\mathbb N}\|\nabla f_j\|_{L^\infty(B_R(p_j))}\ls C_1
 \end{equation}
  for some constant $C_1>0$. Then for any $g_\infty\in  H^1(B_R(p_\infty))$  with $g_\infty-f_\infty\in H^1_0(B_R(p_\infty))$, there exists a sequence of functions $\{g_j\}_{j\in\mathbb N}$ such that
  $g_j-f_j\in H^1_0(B_R(p_j))$, for each $j\in\mathbb N$, and that   $g_j\to g_\infty$ in\ $H^1(B_R(p_j))$.
\end{itemize}

In particular, by taking $f_j\equiv0$, we conclude that: given any $g_\infty\in\  H^1_0(B_R(p_\infty)),$  there exists  a sequence of functions $\{g_j\}_{j\in\mathbb N}$ such that   $g_j\in H^1_0(B_R(p_j))$, for each $j\in\mathbb N$, and that $g_j\to g_\infty$ in $H^1(B_R(p_j))$.
\end{prop}
\begin{proof}
(i). \ \ From the density of the  $Lip_0(B_R(p_j))\subset H^1_0(B_R(p_j))$, we can assume that  $g_j\in Lip_0(B_R(p_j))$ for each $j\in \mathbb N$.

Let $\tilde{g}_j$ be the zero extension of $g_j$ in $X_j$  for each $j\in \mathbb N$. Namely, $\tilde{g}_j=g_j$ in $B_R(p_j)$ and  $\tilde{g}_j=0$ in $X_j\backslash B_R(p_j)$.
Noticing that $\tilde{g}_j$ weakly $L^2$-converge  to $\tilde{g}_\infty$ in the sense of \cite{gms15} (see also \cite{ah16}) and that
 $$\||\nabla \tilde{g}_j|\|_{L^2(X_j)}= \||\nabla g_j|\|_{L^2(B_R(p_j))}\ls C,$$
we obtain that $\tilde{g}_\infty\in H^1(X_\infty)$ and that  $\tilde{g}_\infty=g_\infty$  $\mu_\infty$-a.e.   in $ B_R(p_\infty)$, and that  $\tilde{g}_\infty=0$ $\mu_\infty$-a.e. in  $X\backslash \overline{B_R(p_\infty)}$.

 Now we want to show $g_\infty\in H^1_0(\overline{B_R(p_\infty)})$.      Noting that $\tilde{g}_\infty$ is a $2$-quasi continuous function and that  $ X_\infty\backslash \overline{B_R(p_\infty)}$ is an open set, we conclude, by \cite[Theorem 3.2]{kkm00}, that $\tilde{g}_\infty=0$ $2$-q.e. in $ X_\infty\backslash \overline{B_R(p_\infty)}$.
 By using the fact that  $\tilde{g}_\infty=g_\infty$  $\mu_\infty$-a.e.   in $ B_R(p_\infty)$ and that $\mu_\infty(\partial B_R(p_\infty))=0$, we have  $\tilde{g}_\infty=g_\infty$  $\mu_\infty$-a.e.   in $\overline{ B_R(p_\infty)}$.
 Hence, by Definition~\ref{definition2.4}, we have
$g_\infty\in H^1_0(\overline{B_R(p_\infty)})$.

 At last, by using the assumption  $\partial B_R(p_\infty)=  \partial\big(X_\infty\backslash \overline{B_R(p_\infty)}\big)$ and Corollary \ref{corollary2.6}, we conclude $g_\infty\in H^1_0(B_R(p_\infty))$.

 (ii). \ \  From the density of the  $Lip_0(B_R(p_\infty))\subset H^1_0(B_R(p_\infty))$, we can assume that  $g_\infty-f_\infty\in Lip_0(B_R(p_\infty))$, and hence $g_\infty\in Lip(B_R(p_\infty))$ (since $f_\infty\in Lip(B_R(p_\infty))$. We use Lemma \ref{lemma2.12}  to lift a sequence of functions $\hat g_j\in Lip(B_R(p_j))$ so that:
  \begin{align}\label{equa3.2}
	  \|\hat g_j-g_\infty\circ \Psi_j\|_{L^\infty} &\ls \kappa(\varepsilon_j),\\[1ex]
  \label{equa3.3}
	  \||\nabla \hat g_j|\|_{L^\infty} &\ls \||\nabla g_\infty|\|_{L^\infty} +\kappa(\varepsilon_j)\ls C_{g_\infty},\\[1ex]
   \label{equa3.4}
	  \||\nabla \hat g_j|\|_2 &\ls \||\nabla g_\infty|\|_2+\kappa(\varepsilon_j),
 \end{align}
where $\kappa(\varepsilon_j)$ depends on $K,N,R$ and $\||\nabla g_\infty|\|_{L^\infty}$, and the maps $\Psi_j:\ B_R(p_\infty)\!\to\! B_R(p_j)$ are the $\varepsilon_j$-$mGH$ approximations. By (\ref{equa3.2}) and the facts that $f_j\to f_\infty$ uniformly over $B_R(p_j)$ and that $g_\infty-f_\infty\in Lip_0(B_R(p_\infty))$,  we  get
\begin{equation*}
|\hat g_j(x)-f_j(x)|\ls \kappa_1(\varepsilon_j)\quad {\rm for\ all}\ \   x \ {\rm\ is\ close\ near}\ \   \partial B_R(p_j),
\end{equation*}
for each $j\in \mathbb N$, and for some positive function $\kappa_1$ with $\lim_{t\to0}\kappa_1(t)=0$.
We shall modify $\hat g_j$ slightly to
\begin{equation}\label{equa3.5}
g_j:=
\begin{cases}
\hat g_j-\kappa_1(\varepsilon_j) & {\rm if}\ \  \hat g_j-f_j\gs \kappa_1(\varepsilon_j), \\
 f_j& {\rm if}\ \ |\hat g_j-f_j|\ls \kappa_1(\varepsilon_j), \\
\hat g_j+\kappa_1(\varepsilon_j) & {\rm if}\ \ \hat g_j-f_j\ls-\kappa_1(\varepsilon_j).
\end{cases}
\end{equation}
Then we have, for each $j\in\mathbb N$, that $g_j-f_j\in Lip_0(B_R(p_j))$ and that, by (\ref{equa3.3})--(\ref{equa3.4}) and (\ref{equa3.1}),
\begin{equation}\label{equa3.6}
 \||\nabla   g_j|\|_{L^\infty}\ls C_{g_\infty}+C_1\quad{\rm and}\quad   \limsup_{j\to\infty}\||\nabla   g_j|\|_2\ls \||\nabla g_\infty|\|_2.
 \end{equation}
From (\ref{equa3.5}), we have
$$\|g_j-\hat g_j\|_{L^\infty}\ls  \kappa_1(\varepsilon_j).$$
The combination of this and (\ref{equa3.2}) implies   $g_j\to g_\infty$ uniformly over $B_R(p_j)$.

 By using (\ref{equa3.6}) and the lower semi-continuity of energy, Lemma \ref{lemma2.13}, we conclude that  $\lim_{j\to\infty}\||\nabla   g_j|\|_2= \||\nabla g_\infty|\|_2$. Thus we  finish the proof of (ii).
  The proof is completed.
 \end{proof}

 As a corollary, we have the following convergence for the solutions of Poisson equations living on varying spaces, which is due essentially to \cite{ding02,hon11,xu14,gms15}.
 \begin{cor}\label{cor3.3}
 Let $R>0$.
Let $\{f_j\}_{j\in\mathbb N\cup\{\infty\}}$ and $\{h_j\}_{j\in\mathbb N\cup\{\infty\}}$ be two sequences of functions on $B_R(p_j)$ such that
$$\mathscr Lf_j=h_j\cdot\mu_j,\quad \forall\ j\in\mathbb N,$$
  on $B_{R}(p_j)$  in the sense of distributions.   Suppose  that, for every  $s\in(0,R)$, $f_j\to f_\infty$ uniformly over $B_s(p_j)$, and   $h_j\to h_\infty$ in $L^2(B_s(p_j))$, and   there exists a constant $C_s>0$  such that
\begin{equation}\label{equa3.7}
\sup_{j\in\mathbb N}\|\nabla h_j\|_{L^\infty(B_s(p_j))}\ls C_s.
\end{equation}
Then we have
$\mathscr Lf_\infty=h_\infty\cdot\mu_\infty$ on $B_R(p_\infty)$   in the sense of distributions.
  \end{cor}
\begin{proof} It suffices to show that for any ball $B\subset \subset B_R(p_\infty)$ there holds $\mathscr Lf_\infty=h_\infty\cdot\mu_\infty$ on $B$   in the sense of distributions. Namely,
$f_\infty$ minimizes the functional
$$I_B(f):=\int_B\big(|\nabla f|^2+h_\infty\cdot f\big){\rm d}\mu_\infty$$
among all of $f\in H^1(B)$ such that $f-f_\infty\in H^1_0(B)$.

We will argue  by a contradiction. Suppose not, then there exists a ball $B_r(q_\infty)\subset\subset B_R(p_\infty)$ such that $f_\infty$ is not a minimizer of $I_{B_r(q_\infty)}(f)$. According to \cite[Theorem 7.12]{che99},  there exists a function $g_\infty\in H^1(B_r(q_\infty))$ such that $g_\infty-f_\infty\in H_0^1(B_r(q_\infty))$ and that
\begin{equation}\label{eq3.8}
I_{B_r(q_\infty)}(g_\infty):=\min_{f-f_\infty H_0^1(B_r(q_\infty))}I_{B_r(q_\infty)}(f)\ls I_{B_r(q_\infty)}(f_\infty)-\varepsilon_0
\end{equation}
for some $\varepsilon_0>0$.

Fix some $s_0<R$ such that $B_r(q_\infty)\subset\subset B_{s_0}(p_\infty)$. Take points $B_{s_0}(p_j)\ni q_j\to q_\infty$.  Note that $f_j\in Lip(B_r(q_j))$ for each $j\in \mathbb N\cap\{\infty\}$, and $f_j\to f_\infty$ uniformly over $B_r(q_j)$. Recall that, on each $B_{s_0}(p_j)$, we have $|\nabla h_j|\ls C_1$  and $\mathscr Lf_j=h_j\cdot\mu_j$. Then, the localized Bochner formula \cite[Theorem~3.2]{zz-local} implies that
\begin{align*}
	\frac{1}{2}\mathscr L(|\nabla f_j|^2) &\gs \frac{(h_j)^2}{N}+\ip{\nabla f_j}{\nabla h_j }+K|\nabla f_j|^2\\ &\gs -|\nabla h_j|\cdot|\nabla f_j|+K\cdot|\nabla f_j|^2
\end{align*}
  in the sense of distributions in $B_{s_0}(p_j)$.  Then we by (\ref{equa3.7}) get that
 \begin{equation}\label{equa3.9}
\||\nabla f_j|\|_{L^\infty(B_{r}(q_j))}\ls C_2,\quad \forall\ j\in\mathbb N,
\end{equation}
where $C_2$  depends only on $N,K,R,C_{s_0}$ and
$dist(B_r(q_j),\partial B_{s_0}(p_j)).$

By using Proposition \ref{prop3.2} (ii) on $B_r(q_j)$ and noting that $g_\infty-f_\infty\in H_0^1(B_r(q_\infty))$, we obtain a sequence of functions $g_j\in H^1(B_r(q_j))$  such that $g_j\to g_\infty$ in $H^1(B_r(q_j))$ and that $g_j-f_j\in H^1_0(B_r(q_j))$ for all $j\in\mathbb N$. The combination of   $g_j\overset{H^1}{\to} g_\infty$   and  $h_j\overset{L^2}{\to} h_\infty$  implies
\begin{equation}\label{eq3.10}
 I_{B_r(q_j)}(g_j):=  \int_{B_r(q_j)}\big(|\nabla g_j|^2+ h_j g_j\big)\du_j\to  I_{B_r(q_\infty)}(g_\infty),\quad {\rm as}\ \ j\to\infty.
 \end{equation}
The fact $\mathscr Lf_j=h_j\cdot\mu_j$ on $B_{r}(q_j)$  in the sense of distributions yields
$$I_{B_r(q_j)}(f_j)\ls I_{B_r(q_j)}(g_j)$$
for each $j\in\mathbb N$. Then, we by combining with (\ref{eq3.8}) and (\ref{eq3.10}) have that
\begin{equation*}
\limsup_{j\to\infty}I_{B_r(q_j)}(f_j)\ls I_{B_r(q_\infty)}(f_\infty)-\varepsilon_0.
\end{equation*}
This contradicts to   the lower semi-continuity of energy, Lemma \ref{lemma2.13},
by noticing that  $f_j\to f_\infty$ uniformly over $B_r(q_j)$ and (\ref{equa3.9}), and  $h_j\to h_\infty$ in $L^2(B_r(q_j))$.
  The proof is completed.
\end{proof}

\subsection{Estimates of Dirichlet eigenvalues and eigenfunctions}

 Let $(X,d,\mu)$ be a metric measure space satisfying $RCD^*(K,N)$ for some $K\in \mathbb R$ and $N\in[1,\infty)$. Note that, for any  $N'>N$  and $K'<K$,  $(X,d,\mu)$ satisfies $RCD^*(K',N')$ too. For simplicity, we always assume that $K\ls0$ and $N\gs 3$  in the following.

 Fix a geodesic ball $B_R(p)\subset X$ with radius $R\in(0,{\rm diam}(X)/a)$ for some $a>2$ (hence ${\rm diam} B_R(p)< \frac{{\rm diam}(X)}{a/2}$). Denote by $\lambda_m^{(R)}:=\lambda_m^{B_R(p)}$   the $m-$th Dirichlet eigenvalues of $\Delta_{B_R(p)}$ on ball $B_R(p)$, and by $\phi_m^{(R)}$  the associated eigenfunction with respect to $\lambda_m^{(R)}.$ We normalize $\phi_m^{(R)}$ such that $\|\phi_m^{(R)}\|_2=1$.  The Dirichlet heat kernel on $B_R(p)$ is
$$H^{(R)}(x,y,t)=\sum_{m=1}^\infty e^{-\lambda_m^{(R)}t}\phi_m^{(R)}(x)\phi_m^{(R)}(y).$$

\begin{lem}\label{lem3.4}
Let $B_R(p)$ and $\lambda_m^{(R)},\phi_m^{(R)}$  be as  the above. Then there exist  constants $C'_1,C'_2>0$, depending only on $N,K$ and $R$, such that
\begin{equation}\label{equa3.11}
 C'_1\cdot m^{2/N} \ls \lambda_m^{(R)}\ls  C'_2\cdot m^2,\quad \forall\ m\in\mathbb N.
\end{equation}
\end{lem}
\begin{proof}To simplify the notations, in this proof, we will denote by $B_R:=B_R(p)$ and $\lambda_m:=\lambda_m^{(R)}$.
From the monotonicity of the heat kernels with respect to domains and (\ref{equa2.4})--(\ref{equa2.5}), we have
\begin{equation*}
\begin{split}
	H^{(R)}(x,x,t)\ls H(x,x,t) &\ls \frac{C_{N,K}}{\mu\big(B_{\sqrt t}(x)\big)}\cdot \exp(C_{N,K}\cdot t)\\ &\ls \frac{C_{N,K,R}}{\mu\big(B_{R}(x)\big)}\cdot t^{-N/2}\cdot\exp(C_{N,K}\cdot t).
\end{split}
\end{equation*}
By integrating over $B_R(x)$, we get, for each $m\in\mathbb N$, that
\begin{align*}
	m\cdot e^{-\lambda_mt}\ls \sum_{l\ls m}e^{-\lambda_lt}\ls \sum_{l\in\mathbb N}e^{-\lambda_lt}&\ls C_{N,K,R} \cdot t^{-N/2}\cdot\exp(C_{N,K}\cdot t)\\
	&:=C_{1} \cdot t^{-N/2}\cdot\exp(C_{2}\cdot t) .
\end{align*}
Setting $t=\frac{1}{\lambda_m}$ and noting that $\lambda_m\gs \lambda_1\gs C_{N,K,R}$ for some constant $C_{N,K,R}$ (by  the $L^2$-Poincar\'e inequality on $B_R$), we conclude that
$$m/e\ls C_1\cdot \lambda_m^{N/2}\cdot\exp(C_2/\lambda_m)\ls C_1\cdot \lambda_m^{N/2}\cdot\exp(C_2/\lambda_1) \ls C_3\cdot \lambda_m^{N/2}.$$
This implies the   lower bounds in  (\ref{equa3.11}).

The upper bounds in (\ref{equa3.11}) can be proved by a comparison result for heat kernels of Cheng (see, for example, \cite{ding02}). Here we provide a simple argument as follows.

Fix any $m\in \mathbb N$. We can find $m$ points $\{x_l\}_{l=1,2,\dots,m}$ in $B_R$ such that $d(x_l,x_{l'})\gs R/m$ for any $1\ls l\not=l'\ls m.$ We define functions $\psi_l$ by
 \begin{equation*}
 \psi_l(\cdot):=\eta\big(d(x_l,\cdot)\big),\quad \forall\ l=1,2,\dots,m,
 \end{equation*}
where the function $\eta(s)$ is given by
\begin{equation*}
  \eta(s):=
  \begin{cases}
  1& s\ls \frac{R}{8m}\\
  \frac{8m}{R}\cdot\big(\frac{R}{4m}-s\big) & s\in\left(\frac{R}{8m},\frac{R}{4m}\right)\\
  0& s\gs \frac{R}{4m}.
  \end{cases}
\end{equation*}
It is clear that $\int_{B_R}\psi_l\psi_{l'}\du=0$, for all $ 1\ls l\not=l'\ls m,$
and that, for any $l=1,2,\dots,m$,
$$\frac{\int_{B_R}|\nabla \psi_l|^2\du}{\int_{B_R}\psi_l^2\du}\ls \frac{\big(\frac{8m}{R}\big)^2\cdot \mu\big(B_{\frac{R}{4m} }(x_l)\big)}{\mu\big(B_{\frac{R}{8m} }(x_l)\big)}\ls   \frac{2^{C_D+6}}{R^2}\cdot m^2,$$
where $C_D$ is the doubling constant of $\mu$ on $B_R$, depending only on $N,K$ and $R$.
By the Rellich's compactness (see also \cite[Eq.(5.2)]{bm06}), the  Courant's min-max principle of eigenvalues still holds  (see, for example, \cite{gms15}). It follows  $\lambda_m\ls  2^{C_D+6}\cdot R^{-2}\cdot m^2.$ The proof is completed.
\end{proof}

\begin{lem}\label{lem3.5}
Let $B_R(p)$ and $\lambda_m^{(R)},\phi_m^{(R)}$  be as  the above. Then there exists a constant $C_{N,K,R,m}>0$, depending on $N,K,R$ and $m$, such that
\begin{equation*}
\|\phi_m^{(R)}\|_{L^\infty(B_R(p))}\ls C_{N,K,R}\cdot \lambda_m\ls C_{N,K,R,m},
\end{equation*}
and that
\begin{equation}\label{eq3.12}
\|\nabla \phi_m^{(R)}\|_{L^\infty(B_{r}(p))}\ls \frac{C_{N,K,R,m}}{R-r},\quad \forall \ r\in(0,R).
\end{equation}
\end{lem}
\begin{proof}
To simplify the notations, we shall denote by $B_R:=B_R(p)$, $\lambda_m:=\lambda_m^{(R)}$ and $H(x,y,t):=H^{(R)}(x,y,t).$

From $\Delta \phi_m=-\lambda_m\cdot \phi_m$ and $\lambda_m\gs0$, we can get $\mathscr L|\phi_m|\gs-\lambda_m|\phi_m|$ in the sense of distributions. Noticing that $|\phi_m|\in H^1_0(B_R),$ the Sobolev inequality (see  \cite[Eq.(5.2)]{bm06}) (by the standard argument of Nash-De Giorgi-Moser iteration and $\|\phi_m\|_2=1$, indeed, we can choose the $|\phi_m|$ as the text function) implies
$$\|\phi_m\|_{L^\infty(B_R)}\ls C_{N,K,R}\cdot \lambda_m \cdot\|\phi_m\|_{L^2(B_R)}= C_{N,K,R}\cdot \lambda_m$$
for some constant $C_{N,K,R}>0$.

By using  the equation $\mathscr L\phi_m=\Delta \phi_m\cdot \mu=-\lambda_m \phi_m\cdot\mu$   and the localized Bochner formula \cite[Theorem3.2]{zz-local}, we have
\begin{align*}
	\frac{1}{2}\mathscr L(|\nabla \phi_m|^2)&\gs \frac{(\Delta \phi_m)^2}{N}+\ip{\nabla \phi_m}{\nabla(\Delta \phi_m)}+K|\nabla \phi_m|^2\\
	&\gs (K-\lambda_m)\cdot|\nabla\phi_m|^2
\end{align*}
  in the sense of distributions in $B_R$.  The   Nash-De Giorgi-Moser iteration and Lemma \ref{lem3.4} implies that
\begin{align*}
	\||\nabla \phi_m|\|_{L^\infty(B_r)} &\ls \frac{C'_{N,K,R}}{R-r}\cdot |K-\lambda_m| \cdot\||\nabla \phi_m|\|_{L^2(B_R)}\\
	&=\frac{C'_{N,K,R}}{R-r}\cdot |K-\lambda_m| \cdot\lambda_m\ls \frac{C_{N,K,R,m}}{R-r}
\end{align*}
for some constant $C'_{N,K,R}>0$. The proof is finished.
\end{proof}

\subsection{The convergence of heat kernels}

Let  $K\ls 0$ and $N\gs 3$ and  let $(X_j,d_j,\mu_j)_{j\in\mathbb N\cup\{\infty\}}$ be a  sequence of  metric measure spaces so
   that $(X_j,d_j,\mu_j)$ satisfying $RCD^*(K,N)$ for each $j\in\mathbb N$.
   Take points $p_j\in X_j$, for all $j\in\mathbb N\cup\{\infty\}$. We  assume that
$$ (X_j,d_j,\mu_j,p_j)\overset{pmGH}{\longrightarrow} (X_\infty,d_\infty,\mu_\infty,p_\infty).$$
Hence, the $(X_\infty,d_\infty,\mu_\infty)$ satisfies still $RCD^*(K,N)$ (see \cite{eks15}).

Fix  $a>2$ and $R>0$ with  $R\in(0,{\rm diam}(X_j)/a)$   for all $j\in\mathbb N$.
For each $j\in\mathbb N$,  we denote by $\lambda_{m,j}^{(R)}$  the $m-$th Dirichlet eigenvalues of $\Delta_{B_R(p_j)}$ on the ball $B_R(p_j)$, and by $\phi_{m,j}^{(R)}$, normalized by $\|\phi_{m,j}^{(R)}\|_2=1$,   the associated eigenfunction with respect to $\lambda_{m,j}^{(R)}.$ The Dirichlet heat kernel on $B_R(p_j)$ is
$$H_j^{(R)}(x,y,t)=\sum_{m=1}^\infty e^{-\lambda_{m,j}^{(R)}t}\phi_{m,j}^{(R)}(x)\phi_{m,j}^{(R)}(y).$$

 By using Lemma \ref{lem3.4} and Lemma \ref{lem3.5}, we can assume that, after passing to a subsequence, (say $j_k$,), for each fixed $m\in\mathbb N$, the eigenvalues and eigenfunctions converge:
 \begin{equation}\label{eq3.13}
 \lim_{j\to\infty}\lambda_{m,j}^{(R)}=\lambda_{m,\infty},
\end{equation}
and
 \begin{equation}\label{eq3.14}
\lim_{j\to\infty}\phi_{m,j}^{(R)}=\phi_{m,\infty} .
\end{equation}
 where the convergence of $\phi_{m,j}^{(R)}{\to}\phi_{m,\infty}$ is in   $L^2(B_{R}(p_j))$ and is also uniformly in $B_{r}(p_j)$, for any $r\in(0,R)$, by Lemma \ref{lem3.5} and the Arzela-Ascoli theorem.

\begin{lem}\label{lem3.6}
Assume that  $\partial B_R(p_\infty)=  \partial\big(X_\infty\backslash \overline{B_R(p_\infty)}\big)$.
Let $\Delta^{(R)}_{\infty}:=\ \Delta_{B_R(p_\infty)}$ be the  infinitesimal generator of the Dirichlet form
$$\big(\mathscr E_{B_R(p_\infty)},H^1_0(B_R(p_\infty))\big)$$ on $B_R(p_\infty)$, with  domain  $ D(\Delta^{(R)}_{\infty})$.
 Then,
for each $m\in\mathbb N$, we have  that   $\phi_{m,\infty}\in  D(\Delta^{(R)}_{\infty})$  and that
\begin{equation}\label{eq3.15}
\Delta^{(R)}_{\infty} \phi_{m,\infty}=-\lambda_{m,\infty}\phi_{m,\infty}.
\end{equation}
That is, $\lambda_{m,\infty}$ is an eigenvalue of $\Delta^{(R)}_{\infty}$ with an associated eigenfunction $\phi_{m,\infty}$. Moreover, the convergence  $\phi_{m,j}^{(R)} \to \phi_{m,\infty}$ in (\ref{eq3.14}) is also in\   $H^1(B_{R}(p_j))$.
\end{lem}

\begin{proof}
From $\phi_{m,j}^{(R)}\to\phi_{m,\infty}$ in $L^2(B_R(p_j))$ and
$$\||\nabla \phi_{m,j}^{(R)}|\|_2=\lambda_{m,j}^{(R)}\|\phi_{m,j}^{(R)}\|_2=\lambda_{m,j}^{(R)},$$
by using Lemma \ref{lem3.4} and Proposition \ref{prop3.2} (i), we have $\phi_{m,\infty}\in H^1_0(B_R(p_\infty)).$
By applying Corollary \ref{cor3.3} and (\ref{eq3.12}), we conclude that
\begin{equation}\label{eq3.16}
\mathscr L_\infty\phi_{m,\infty}=-\lambda_{m,\infty}\cdot\phi_{m,\infty}\cdot\mu_\infty,
\end{equation}
where $\mathscr L_\infty$ is the distributional Laplacian on $B_R(p_\infty).$ Notice that $\phi_{m,\infty}\in H^1_0(B_R(p_\infty))$. From \cite{gig15}, we conclude
that   $\phi_{m,\infty}\in  D(\Delta^{(R)}_{\infty})$ and that (\ref{eq3.15}) holds. This proves the first assertion.

For the second assertion, we need only to show
\begin{equation}\label{eq3.17}
 \||\nabla \phi_{m,j}^{(R)}|\|_2\to\||\nabla \phi_{m,\infty}^{(R)}|\|_2\quad {\rm as} \ \ j\to\infty.
 \end{equation} The equation (\ref{eq3.15}) implies that $\||\nabla \phi_{m,\infty}^{(R)}|\|_2=\lambda_{m,\infty}$. Now the desired (\ref{eq3.17}) comes from the combination of this and   $\||\nabla \phi_{m,j}^{(R)}|\|_2=\lambda_{m,j}^{(R)}$ and $\lambda_{m,j}^{(R)}\to \lambda_{m,\infty}^{(R)}$ as $j\to\infty.$ The proof is finished.
\end{proof}

The following  is the crucial point in this section.
\begin{lem}\label{lem3.7}
Assume that  $\partial B_R(p_\infty)=  \partial\big(X_\infty\backslash \overline{B_R(p_\infty)}\big)$.
The sequence\ $\{\phi_{m,\infty}\}_{m\in\mathbb N}$ forms a complete basis of $L^2(B_R(p_\infty)).$
\end{lem}

\begin{proof}
According to Lemma \ref{lem3.6}, it suffices to show that all of eigenfunctions of  $\Delta^{(R)}_{\infty}$  are in $\{\phi_{m,\infty}\}_{m\in\mathbb N}$. Suppose not,  then there exists  an eigenfunction $\psi_\infty\in H^1_0(B_R(p_\infty))$ with
\begin{equation}\label{eq3.18}
\|\psi_\infty\|_2=1\quad {\rm and}\quad \int_{B_R(p_\infty)}\psi_\infty\cdot\phi_{m,\infty}\du_\infty=0,\ \ \ \forall m\in\mathbb N.
\end{equation}
Let $\sigma_\infty$ be the eigenvalue of  $\Delta^{(R)}_{\infty}$ with respect to $\psi_\infty$. Define
\begin{equation}\label{eq3.19}
m_0:=\max\big\{m\in\mathbb N:\ \lambda_{m,\infty}\ls 2\sigma_\infty+2\big\}.
\end{equation}

By Proposition \ref{prop3.2}  (ii), we can lift $\psi_\infty$ to a sequence of functions $\psi_j\in H^1_0(B_R(p_j))$ such that $\psi_j\to\psi_\infty$ in $H^1(B_R(p_j)).$
For each $j$, since $\{\phi_{m,j}\}^\infty_{m=1}$ is a complete basis in $L^2$, we  denote the Fourier expansion of $\psi_j$ w.r.t. $\{\phi_{m,j}\}$  by
\begin{equation}\label{eq3.20}
\psi_j=\sum_{m=1}^\infty a_{m,j}\phi_{m,j},\quad {\rm where}\quad a_{m,j}:=\int_{B_R(p_j)}\psi_j\cdot\phi_{m,j}\du_j.
\end{equation}
Then
\begin{align}\label{eq3.21}
	\|\psi_j\|^2_2=\sum_{m\gs1}a^2_{m,j}&=\sum_{m\ls m_0}a^2_{m,j}+\sum_{m\gs m_0+1}^{\infty}a^2_{m,j},\\[1ex]
\label{eq3.22}
	\||\nabla \psi_j|\|^2_2&=\sum_{m\gs1}\lambda_{m,j}a^2_{m,j} \gs \sum_{m\gs m_0+1}\lambda_{m,j}a^2_{m,j}\\
	\notag &\gs \lambda_{m_0+1,j}\sum_{m\gs m_0+1}a^2_{m,j}.
\end{align}
By $\psi_j\to\psi_\infty$ and $\phi_{m,j}\to\phi_{m,\infty}$ in $L^2(B_R(p_j))$ as $j\to\infty,$ we have, for each $m\in \mathbb N$, that
$$\lim_{j\to\infty} a_{m,j}=\int_{B_R(p_\infty)}\psi_\infty\cdot\phi_{m,\infty}\du_\infty=0.$$
By combining with $\psi_j\to\psi_\infty$ in $H^1(B_R(p_j))$ and  $\lambda_{m,j}\to \lambda_{m,\infty}$ as $j\to\infty$, for any  given $\epsilon\in (0,1)$, there exists some $j_0=j_0(m_0,\epsilon)>0$ such that for all $j\gs j_0$ we have that
\begin{equation}\label{eq3.23}
\|\psi_j\|_2^2\gs 1-\epsilon, \quad \||\nabla \psi_j|\|_2^2\ls\sigma_\infty +\epsilon, \quad |a_{m,j}|\ls \epsilon,  \quad   \forall\ m\ls m_0+1,
\end{equation}
(where we have used $\|\psi_{\infty}\|_2=1$ and $\|\nabla \psi_{\infty}\|^2_2=\sigma_\infty$,) and that
\begin{equation}\label{eq3.24}
 \quad \ \lambda_{m_0+1,j}\gs \lambda_{m_0+1,\infty}-\epsilon\overset {(\ref{eq3.19})}\gs 2\sigma_\infty+1.
\end{equation}
  From (\ref{eq3.21})-(\ref{eq3.22}) and (\ref{eq3.23}), we get that, for all $j\gs j_0$,
\begin{equation*}
\qquad \sum_{m\gs m_0+1}a^2_{m,j}\gs 1-\epsilon-m_0 \epsilon^2,\quad \lambda_{m_0+1,j}\sum_{m\gs m_0+1}a^2_{m,j}\ls   \sigma_\infty+\epsilon.
\end{equation*}
The combination of  this and (\ref{eq3.24})  implies that
 \begin{equation*}
 1-\epsilon-m_0 \epsilon^2 \ls \frac{\sigma_\infty+\epsilon}{2\sigma_\infty+1}.
 \end{equation*}
This is impossible when $\epsilon$ is small enough. The proof is finished.
 \end{proof}
This Lemma states that $\{\lambda_{m,\infty},\phi_{m,\infty}\}_{m\in\mathbb N}$ is the \emph{complete} spectral system of  $\Delta^{(R)}_{\infty}$. Hence, the limit of the convergence in (\ref{eq3.13}) is unique and does not  depend  of the choice of subsequence $j_k$. So the Dirichlet heat kernel of $B_{R}(p_\infty)$, in fact, is
\begin{equation}\label{eq3.25}
H^{(R)}_{\infty}(x,y,t) =\sum_{m=1}e^{-\lambda_{m,\infty}t} \phi_{m,\infty}(x) \cdot\phi_{m,\infty}(y).
\end{equation}
Therefore, we have obtained:
\begin{thm}\label{thm3.8}
\hspace{-1ex}Assume that  $\partial B_R(p_\infty)\!=\!  \partial\big(X_\infty\backslash \overline{B_R(p_\infty)}\big).$\footnote{According to Remark \ref{rem2.4}(1), the assumption can be replaced by $${\rm Cap}_2\big (\partial B_R(p_\infty)\backslash  \partial\big(X_\infty\backslash \overline{B_R(p_\infty)}\big)\big)=0.$$} \
Let  $H^{(R)}_j(x,y,t)$ be the Dirichlet heat kernel on $B_R(p_j)$ for all $j\in\mathbb N\cup\{\infty\}$.
Then for any fixed  $t>0$,  we have
that $H^{(R)}_j(p_j,\cdot,t)\to H^{(R)}_\infty(p_\infty,\cdot,t)$ is in $L^2(B_R(p_j))$ and is also uniformly in $B_{R/2}(p_j)$, as $j\to\infty$.

In particular, the local spectral convergence, Proposition \ref{prop1.3}, holds.
\end{thm}
Remark that an $L^2$-convergence theorem for  global  heat flows on\ $(X_j,d_j,\mu_j)$ was proved in \cite{gms15}.  An $H^1$-convergence theorem for local heat flows has been recently obtained by Ambrosio-Honda in \cite{ah17}.

\begin{lem}\label{lem3.9}
Let $(X,d,\mu,p)$ be a pointed metric measure space with\ $RCD^*(K,N)$ for some $K\in\mathbb R$ and $N\gs 3.$ Let  $H^{(R)}(x,y,t)$ be
the Dirichlet heat kernel on $B_R(p)$. Then, for any $t>0$, there exists a constant  some constant $C_{N,K,t}>0$  such that  for all $R'\gs R\gs R_0:=\max\big\{5t, \sqrt{5Nt/2}\big\}$, we have that
\begin{equation}\label{equa3.26}
\sup_{B_R(p)}\big|H^{(R')}(x,p,t)-H^{(R)}(x,p,t)\big|\ls \frac{C_{K,N,t}}{\mu(B_{\sqrt t}(p))}\cdot e^{-R},
\end{equation}
and that
\begin{equation}\label{equa3.27}
\int_{X\backslash B_R}H^{(R')}(x,p,t)\du(x)  \ls C_{N,K,t}\cdot e^{-R}.
\end{equation}
\end{lem}

\begin{proof}
Let $H(x,y,t)$  be the heat kernel on $(X,d,\mu)$. Recalling the monotonicity of heat kernels with respect to domains that
\begin{equation}\label{equa3.28}
  H^{(R)}(x,p,t)\ls H^{(R')}(x,p,t)\ls H(x,p,t),\quad \forall\ x\in B_R(p),
\end{equation}
 and combining with the upper bound of $H$, (\ref{equa2.5}), we have
 \begin{align}\label{equa3.29}
	&\quad\ \sup_{B_R}\big|H(x,p,t)-H^{(R)}(x,p,t)\big|\\
	\notag &\ls \sup_{\partial B_R\times(0,t] }\big|H(x,p,s)-H^{(R)}(x,p,s)\big|\\
\notag &\ls  \sup_{s\in(0,t] }\frac{C_{K,N}}{\mu(B_{\sqrt s})}\cdot \exp\left(-\frac{R^2}{5s}+C_{K,N}\cdot s\right)\\
\notag &\ls  C_{K,N}\cdot \exp\Big(C_{K,N}\cdot t\Big)\cdot \sup_{s\in(0,t] }\frac{1}{\mu(B_{\sqrt s})}\cdot \exp\left(-\frac{R^2}{5s}\right),
\end{align}
 where  $B_r:=B_r(p)$, and we have used the maximum principle for the first inequality, since both $H(\cdot, p,\cdot)$ and $H^{(R)}(\cdot, p,\cdot)$ are  weak solutions of the heat equation on $B_R\times(0,\infty)$ with the same initial data.
 From the local measure doubling property (\ref{equa2.4}), we have
\begin{equation}\label{equa3.30}
\frac{\mu(B_{\sqrt t}) }{\mu(B_{\sqrt s})}\ls C'_{N,K,t}\Big(\frac{t}{s}\Big)^{N/2},\ \quad \forall\ 0<s<t,
\end{equation}
for some constant $C'_{K,N,t}>0.$
 Let us put
$$v(s):=\frac{\exp\Big(-\frac{R^2}{5s}\Big)}{s^{N/2}},\ \quad\forall\ s\in(0,t].$$
If $R^2\gs  5Nt/2 $, then we by $\frac{{\rm d}}{{\rm d}s}\log v(s)=\frac{R^2}{5s^2}-\frac{N}{2s}=\frac{1}{5s^2}\big( R^2-\frac{5Ns}{2}\big)\gs0$ have that $v(s)\ls v(t)$ for all $0<s<t$. Hence,
\begin{equation}\label{equa3.31}
\sup_{s\in(0,t]}\frac{\exp\Big(-\frac{R^2}{5s}\Big)}{s^{N/2}}\ls \frac{\exp\Big(-\frac{R^2}{5t}\Big)}{t^{N/2}},
\end{equation}
provided $R \gs \sqrt{ 5Nt/2 }$. The combination of (\ref{equa3.29})--(\ref{equa3.31}) yields that
\begin{align}\label{equa3.32}
&\quad\ \sup_{B_R}\big|H(x,p,t)-H^{(R)}(x,p,t)\big|\\
\notag &\ls  C_{K,N}\cdot \exp\Big(C_{K,N}\cdot t\Big)\cdot\frac{C'_{N,K,t}}{\mu(B_{\sqrt t})}\cdot  \exp\Big(-\frac{R^2}{5t}\Big)\\
\notag &\ls\frac{C_{N,K,t}}{\mu(B_{\sqrt t})}\cdot e^{-R},\ \quad ({\rm by}\ \ R\gs 5t)
\end{align}
for any $R\gs R_0:=\max\big\{5t, \sqrt{5Nt/2}\big\}$. Now the assertion (\ref{equa3.26}) comes from (\ref{equa3.28}) and (\ref{equa3.32}).

The generalized Bishop-Gromov inequality (\ref{equa2.2}) for  $RCD^*(K,N)$-space implies that, for all $j\gs1$ and all $R\gs \sqrt t$,
\begin{equation*}\label{equa3.33}
\frac{\mu(B_{(j+1)R}\backslash B_{jR})}{\mu(B_{\sqrt t})}\ls \frac{\overline{\mu}\big((j+1)R\big)-\overline{\mu}(jR)}{\overline{\mu}(\sqrt t)}\ls\frac{e^{C'_{N,K}\cdot jR}}{\overline{\mu}(\sqrt t)}
\end{equation*}
for some constant $C'_{N,K},$ where $\overline{\mu}(s):=\int^s_0\mathfrak{s}^{N-1}_{\frac{K}{N-1} }(\tau){\rm d}\tau$ and   $\mathfrak{s}_k(\tau)$ is given in (\ref{equa2.3}).
Hence we have, for all $j\gs1$ and all $R\gs \sqrt t$ (which is ensured by $R\gs R_0$), that
\begin{align*}
&\quad\ \int_{X\backslash B_R}H(x,p,t)\du(x)\\
& \ls \frac{C_{K,N}}{\mu(B_{\sqrt t})}\cdot \sum_{j=1}^\infty\int_{ B_{(j+1)R}\backslash B_{jR}}\exp\left(-\frac{d^2(p,x)}{5t}+C_{K,N}\cdot t\right)\du(x)\\
&\ls C_{K,N} \sum_{j=1}^\infty \frac{\mu(B_{(j+1)R}\backslash B_{jR})}{\mu(B_{\sqrt t})}\cdot \exp\left(-\frac{(jR)^2}{5t}+C_{N,K}\cdot t\right) \\
&\ls  \frac{ C_{K,N}}{\overline{\mu}(\sqrt t)} \sum_{j=1}^\infty\cdot \exp\left(-\frac{(jR)^2}{10t}+C_{N,K}\cdot t+C'_{N,K}\cdot jR \right).
\end{align*}
The combination  of this and (\ref{equa3.28}) implies (\ref{equa3.27}). The proof is finished.
\end{proof}

As a consequence, we have the convergence of heat kernels as follows.
\begin{cor}\label{cor3.10}
For any $R$ large enough, we assume that  $\partial B_R(p_\infty)=  \partial\big(X_\infty\backslash \overline{B_R(p_\infty)}\big)$.

Let $R_j$ be a sequence such that $R_j\to\infty$ as $j\to\infty.$ Let $H^{(R_j)}_j(x,y,t)$ be the Dirichlet heat kernel on $B_{R_j}(p_j)$ for all $j\in\mathbb N\cup\{\infty\}$. Then for any fixed  $t>0$,  the convergence $H^{(R_j)}_j(\cdot,p_j,t)\to H_\infty(\cdot,p_\infty,t)$ holds in $L^1$ sense, and holds also uniformly in $B_{R}(p_j)$ for any fixed $R$, as $j\to\infty$.
\end{cor}
\begin{proof}
It comes immediately from the combination of Theorem \ref{thm3.8} and Lemma \ref{lem3.9}.
\end{proof}
This corollary in the special case where $(M^n_j,g_j,vol_j,p_j)$ are a sequence\ of smooth $n$-dimensional Riemannian manifolds with $Ric\gs -k$ and\ $vol_j(B_1(p_1))\gs v_0>0$, was earlier  obtained by Ding \cite{ding02}.  A  pointed converging theorem for global heat kernels on $RCD^*(K,N)$ spaces was recently given in \cite{aht17}.

\begin{rem}\label{rem3.11}
Let $X:=C(Y)$ be the cone over space $Y$ such that\ $(X,d,\mu,o_Y)$ is an $RCD^*(0,N)$ space (with the cone metric and cone measure), where $o_Y$ is the vertex. Then, for any ball $B_R(o_Y)$, we have  $\partial B_R(o_Y)=  \partial\big(X\backslash \overline{B_R(o_Y)}\big)$. In particular, it holds for any  Euclidean space.
\end{rem}

\section{Weyl's law}

In this section, we fix a metric measure space $(X,d,\mu)$ satisfying\ $RCD^*(K,N)$ for some $K\in\mathbb R$ and $N\in[1,\infty).$ Without loss the generality, we can assume that $K\ls0$  and $N\gs3$ in the following.

Let $ p\in X$ and  $r\in(0,1)$, we consider the rescaled and normalized pointed metric measure space $(X,d_r,\mu_{r}^{  p},  p)$, where
\begin{equation}\label{eq4.1}
\begin{aligned}
&d_r(\cdot,\cdot):=r^{-1}d(\cdot,\cdot),\quad  \mu_{r}^{  p}:=\frac{\mu}{b(  p,r)}\quad {\rm and}\\
&b(  p,r):=\int_{B_{r}(  p)}\left(1-\frac{d(  p,x)}{r}\right)\du(x).
\end{aligned}
\end{equation}
By the measure doubling property, we have $$\mu\big(B_{r}(p)\big)\gs b(p,r)\gs \frac{1}{2C_D}\mu\big(B_{r}(p)\big),$$
where $C_D$ is the doubling constant on $B_r(p)$. Indeed,   for any $r>0$, we have
\begin{equation*}
\begin{split}
 \int_{B_r(p)}d(p,x)\du(x)&\ls \frac{r}{2}\cdot \mu\big(B_{r/2}(p)\big) +r\cdot \mu\big(B_{r}(p)\backslash B_{r/2}(p)\big)  \\ &=r\cdot \mu\big(B_{r}(p)\big)-\frac{r}{2}\cdot\mu\big(B_{r/2}(p)\big)\ls \left(r-\frac{r}{2C_D}\right)\mu\big(B_{r}(p)\big).
\end{split}
 \end{equation*}
This implies immediately that
$$b(p,r)=\mu\big(B_{r}(p)\big)-\frac 1 r\int_{B_r(p)}d(p,x)\du(x) \gs  \frac{1}{2C_D} \mu\big(B_{r}(p)\big).$$

 \begin{defn}[Tangent cones]\label{definition4.1}
\emph{ Let $(X,d,\mu)$ be a metric measure space and let $  p\in X$. A pointed metric measure space $(Y,d_Y,{\rm m}_{Y}, y)$ is called a \emph{ tangent cone of  $(X,d,\mu)$ at $ p$} if there exists a sequence   $\{r_j\}_{j\in\mathbb N}$ with $r_j\to 0$, as $j\to\infty$, such that
 $$(X,d_{r_j},\mu_{r_j}^{  p}, p)\overset{pmGH}{\to}(Y,d_Y,{\rm m}_{Y}, y).$$
 The set of all the tangent cones at $  p$  is denoted by ${\rm Tan}(X,d,\mu,  p).$ Remark that a tangent cone at $ p$  may depend on the choice of the sequence  $\{r_j\}$.}
\end{defn}
A point $ p\in X$, is  called a $k$-regular point if the tangent cones at $p$ is unique and if
\begin{equation}\label{eq4.2}
{\rm Tan}(X,d,\mu,  p)=\big\{\big(\mathbb R^{k},d_{\rm E}, \mathscr L_{k},0\big)\big\},
\end{equation}
   where
$d_{\rm E}$ is the standard Euclidean metric of $\mathbb R^k$ and $\mathscr L^k$ is the $k$-dimensional Lebesgue measure normalized so that
$\int_{B_1(0)}(1- |x|)d\mathscr L^k(x) = 1.$ We denote by $\mathcal R_k:=$ all of $k$-regular points of $(X,d,\mu)$.

Very recently, a structure theorem of  $RCD^*(K,N)$-spaces has been given by Mondino-Naber \cite{mn14}, and  by Kell-Mondino \cite{km16},  Gigli-Pasqualetto \cite{gp16} and De Philippis al. \cite{dpmr16}.
\begin{thm}\label{thm4.2}
Let $(X,d,\mu)$ be a metric measure space satisfying\ $RCD^*(K,N)$ for some $K\in\mathbb R$ and some $N\in[1,\infty)$. Then we have\\
\indent\emph{ (i) (\cite[Theorem 6.7]{mn14})}. $\mu\big(X\backslash \cup_{1\ls k\ls [N]}\mathcal R_k\big)=0$, where $[N]:=\max\{n\in\mathbb N:\ n\ls N\}$;\\
\indent\emph{ (ii) (\cite[Theorem 1.3]{mn14})}. Each $\mathcal R_k$ is $k$-rectifiable. More precisely, for every $\epsilon>0$, we can cover $\mathcal R_k$, up to an $\mu$-negligible subset, by a countable collection of sets $U^{k,\ell}_\epsilon$, $\ell\in\mathbb N,$ with the property that  each    $U^{k,\ell}_\epsilon$ is $(1+\epsilon)$-bilipschitz to a subset of $\mathbb R^k$; \\
\indent\emph{ (iii) (\cite{km16,gp16,dpmr16})}. For each $U^{k,\ell}_\epsilon$ in above (ii), the measure $\mu|_{U^{k,\ell}_\epsilon}\ll\mathscr H^k$, the $k$-dimensional Hausdorff measure.
\end{thm}

Let us recall that the \emph{$k$-dimensional density function} of $\mu$ ($k\gs1$), $\theta_{k}: X\to [0,\infty]$ is defined by
\begin{equation}\label{eq4.3}
\theta_k(p)=\theta_k(\mu,p):  =\lim_{r\to0}\frac{\mu\big(B_r(p)\big)}{\omega_k\cdot  r^k}.
\end{equation}
where  $\omega_k$ is the volume of the unit ball in $\mathbb R^k$ (under the standard Lebesgue's measure).
From Theorem 4.2 (ii) and (iii), we conclude  that, for $\mu$-almost all $p\in \mathcal R_k$, the limit  (\ref{eq4.3}) exists and is in $(0,\infty)$, and that
\begin{equation}\label{equa4.4}
\mu|_{\mathcal R_k}=\theta_k\cdot\mathscr H^k.
\end{equation}
Indeed, by the fact $\mu|_{\mathcal R_k}$ is a $k$-rectifiable measure, we by \cite[Theorem 5.4]{ak00} have that $\mu|_{\mathcal R_k}=\widetilde{\theta}_k\cdot\mathscr H^k$ for some non-negative $\mathscr H^k$-integrable function $\widetilde{\theta}_k$, and moreover,
$$\widetilde{\theta}_k(p)= \lim_{r\to0}\frac{\mu(B_r(p)\cap \mathcal R_k)}{\omega_k\cdot r^k},\quad \mathscr H^k-{\rm a.e.}\ p\in \mathcal R_k.$$
On the other hand, $\mathcal R_k$ has density 1 $\mu$-almost all $p\in \mathcal R_k$, that is $$\lim_{r\to0}\frac{\mu(B_r(p)\cap \mathcal R_k)}{\mu(B_r(p))}=1,\quad \mu-{\rm a.e.}\ \ p\in \mathcal R_k.$$
By using the fact $\mu|_{\mathcal R_k}\ll \mathscr H^k$ again and combining with the above two equalities, we get that, for $\mu$-almost all $p\in \mathcal R_k$, the limit of (\ref{eq4.3}) exists and $\theta_k(p)=\widetilde{\theta}_k(p)$ for $\mu$-almost all points in $\mathcal R_k$ (see also \cite[Theorem 2.12]{dg18}).

\begin{lem}\label{lem4.3}
For   $\mu$-almost all $p\in \mathcal R_k$, we have
\begin{equation}\label{equa4.5}
\lim_{r\to0}\frac{b(p,r)}{r^{k}}=\frac{\theta_k(p)\omega_k}{k+1}.
\end{equation}
\end{lem}
\begin{proof}Let $p\in\mathcal R_k$ such that the limit of  (\ref{eq4.3}) exists and is in  $(0,\infty).$
By (\ref{eq4.3}), we have
\begin{equation}\label{equa4.6}
 \mu\big(B_{r}(p)\big)=\theta_k(p)\omega_k \cdot r^k\cdot \big(1+o(1)\big).
\end{equation}
From (\ref{equa2.2}), it is clear that $r\mapsto\mu\big(B_{r}(p)\big)$ is locally Lipschitz on $(0,R)$ for any $R>0$. So we get  by (\ref{equa4.6})  that for almost all $r\in(0,R)$,
 $$\frac{d}{dr} \mu\big(B_{r}(p)\big):=A_p(r)=  \theta_k(p)\omega_k\cdot k\cdot r^{k-1}\cdot\big(1+o(1)\big),$$
and hence
$$ \int_{B_r(p)}d(p,x)\du(x)=\int_0^rs\cdot A_p(s){\rm d}s=\frac{k\cdot \theta_k(p)\omega_k}{k+1}\cdot r^{k+1}\cdot \big(1+o(1)\big).$$
Therefore, from   (\ref{equa4.6}) and the definition of $b(p,r)$ in (\ref{eq4.1}), we conclude
\begin{equation*}
b(p,r)=\frac{\theta_k(p)\omega_k}{k+1}\cdot r^{k}\cdot \big(1+o(1)\big).
\end{equation*}
This is (\ref{equa4.5}), and  the proof is finished.
\end{proof}

Let $\Omega\subset X $ be a bounded open subset and let $H^{\Omega}(x,y,t)$ be the Dirichlet heat kernel on $\Omega$.
\begin{lem}\label{lem4.4}
For $\mu$-almost all   $p\in \mathcal R_k\cap \Omega$, we have
\begin{equation}\label{equa4.7}
\lim_{t\to0}H^\Omega(p,p,t)\cdot t^{k/2}= \frac{1}{\theta_k(p)\cdot(4\pi)^{k/2}}.
\end{equation}
\end{lem}
\begin{proof}
By Lemma \ref{lem4.3}, it suffices to show that, for $\mu$-almost all   $p\in \mathcal R_k\cap \Omega$,
\begin{equation}\label{equa4.8}
\lim_{t\to0}H^{\Omega}(p,p,t)\cdot b(p,\sqrt t)=\frac{\omega_k}{k+1}\cdot (4\pi)^{-k/2}.
\end{equation}

We shall first consider the case where $\Omega=B_R(p)$ for some $R>0.$

Given any $\alpha,\beta>0$, we denote by $H_{\alpha,\beta}^{(\alpha R)}(x,y,t)$ the Dirichlet heat kernel on $B_{\alpha R}(p)\subset (X,\alpha d,\beta  \mu,p),$ the rescaled space. It is clear that
\begin{equation}\label{equa4.9}
H_{\alpha,\beta}^{(\alpha R)}(x,y,t)=\frac 1 \beta \cdot H^{(R)}(x,y,\frac{t}{\alpha^2}).
\end{equation}
By taking any sequence $r_j\searrow0$ as $j\to\infty$ and choosing
$\alpha:=r^{-1}_j$, $\beta:=b^{-1}(p,r_j),$ by using Corollary \ref{cor3.10} (and Remark \ref{rem3.11}) and  the definition of $\mathcal R_k$, we get
\begin{equation*}
\lim_{j\to\infty} H^{(R)}(p,p,r^2_j)\cdot b(p,r_j)=\lim_{j\to\infty}H_{\alpha,\beta}^{(\alpha R)}(p,p,1)=b_k\cdot (4\pi)^{-k/2},
\end{equation*}
where
$$  b_k:=\int_{B_1(0^k)}\big(1-|x|\big){\rm d}x=\frac{\omega_k}{k+1}.$$
Hence, we get
\begin{equation}\label{eq4.10}
\lim_{t\to0}H^{(R)}(p,p,t)\cdot b(p,\sqrt t) =\frac{\omega_k}{k+1}\cdot (4\pi)^{-k/2}.
\end{equation}

Secondly, we consider that $\Omega$ is general a bounded domain. In this case, we can find two balls such that $B_{R_1}(p)\subset\Omega\subset B_{R_2}(p)$. According to the monotonicity of Dirichlet heat kernels on domains, we have
$$H^{(R_1)}(p,p,t)\ls H^{\Omega}(p,p,t)\ls H^{(R_2)}(p,p,t).$$
The desired  (\ref{equa4.8}) comes from the combination of  this and (\ref{eq4.10}). Now, the proof is finished.
\end{proof}

In order to state the Weyl's law, we   introduce the following condition.
\begin{defn}\label{defn4.5}
\emph{Let  $k_0>0$ and $r_0\in(0,1)$, and let $\Omega\subset (X,d)$ be domain. A Radon measure $\nu$ is said to be   \emph{$k_0-$noncollapsing on $\Omega$ of scale $r_0$}, if the functions $x\mapsto\frac{r^{k_0}}{\nu (B_r(x))}$ are integrable uniformly in $r\in(0,r_0)$. More precisely,
for any $\varepsilon>0$, there exists a constant  $\delta=\delta(\varepsilon,r_0,k_0,\Omega)>0$  such that, for any $\nu$-measurable set $E\subset \Omega$ with $\nu(E)<\delta$, it holds}
\begin{equation}\label{eq4.11}
\int_E \frac{r^{k_0}}{\nu\big(B_r(x)\big)}{\rm d}\nu(x)<\varepsilon,\quad \forall\ r\in(0,r_0).
\end{equation}
\end{defn}
It is clear that if  a measure $\nu$ is $k_0-$noncollapsing on $\Omega$ of scale  $r_0$, then for any $k'\gs k_0 $ and $r'\ls r_0$, it  is still $k'-$noncollapsing of scale  $r'$. Let us consider some examples.

\medskip
\noindent \textbf{{Example  1.}}\ Let $(X,d,\mu)$ be an $RCD^*(K,N)$-space for some $K\in \mathbb R$ and $N\gs1$. Then, for any bounded domain $\Omega\subset X$, the measure $\mu$  must be $N-$noncollapsing on  $\Omega$ of scale $d_\Omega:={\rm diam}(\Omega)$.
Indeed, by (\ref{equa2.4}), we have
$$\frac{\mu(B_{r}(x))}{\mu( B_{d_\Omega}(x))} \gs C_{d_\Omega} r^{N},\quad \forall\ r\in(0,d_\Omega).$$
That is, $\frac{r^N}{\mu(B_{r}(x))} \ls C^{-1}_{d_\Omega} /\mu( B_{d_\Omega}(x)).$ Hence, the measure $\mu$ is  $N-$noncol\-lapsing on  $\Omega$ of scale $d_\Omega$.

\medskip
\noindent\textbf{{Example  2.}}\ Let $(X,d,\mu)$ be a metric measure space. Suppose that  $\mu$ is $k_0-$noncollapsing on $\Omega$ of scale $r_0$. Assume that $\nu$ is another Radon measure  such that $$\nu(B_r(p))\gs C\cdot \mu(B_r(p)),\quad \forall \ r\in (0,r_0),\forall\ p\in \Omega$$
for some constant $C>0$.  Then  $\nu$ is also   $k_0-$noncollapsing on $\Omega$ of scale $r_0$.

\begin{thm}\label{thm4.6}
 Let $(X,d,\mu)$ be an $RCD^*(K,N)$-metric measure space and let $\Omega\subset X$ be a bounded domain. We set
\begin{equation}\label{equa4.12}
k_{{\rm max}}:=\max\big\{k|\ \mu\big(\Omega\cap\mathcal R_k\big)>0\big\}.
\end{equation}
  Assume that $\mu$ is $k_0$-noncollapsing on $\Omega$ of scale $r_0$  for some $k_0\in (0,k_{\rm max}]$ and some $r_0>0$.  Then we have the asymptotic formula of Dirichlet eigenvalues
 \begin{equation}\label{equa4.13}
 \lim_{\lambda\to\infty}\frac{N_\Omega(\lambda)}{\lambda^{k_{\rm max}/2}} =\Gamma(k_{\rm max}/2+1)^{-1}\cdot\frac{\mathscr H^{k_{\rm max}}(\Omega\cap \mathcal R_{k_{\rm max}}) }{(4\pi)^{k_{\rm max}/2}},
 \end{equation}
  where  $N_\Omega(\lambda):=\#\{\lambda_j^\Omega:\ \lambda_j^\Omega\ls \lambda\}$, and $\Gamma(s)$ is the Gamma function.
\end{thm}

\begin{proof}
By the upper bounds of the heat kernel  (\ref{equa2.5}), we have that, for any $t\ls1$,
\begin{equation}\label{equa4.14}
H^\Omega(x,x,t)\ls H(x,x,t)\ls \frac{C_{N,K}}{\mu(B_{\sqrt t}(x))}\cdot\exp\big(C_{N,K}\cdot t\big)\ls\frac{C'_{N,K}}{\mu(B_{\sqrt t}(x))}.
\end{equation}
In the following, we denote by $k_m:= k_{\rm max}$. Since $k_m\gs k_0$, we know that $\mu$ is also $k_m-$noncollapsing on $\Omega$ of scale $r_0$. Fix any $\varepsilon>0$ and let $\delta>0$ be given in the definition of $k_m-$noncollapsing on $\Omega$ of scale $r_0$. From Lemma \ref{lem4.4} and  Egorov's Theorem, there exists a $\mu$-measurable  set $E$ with $\mu(E)<\delta$ such that
 $$H^\Omega(x,x,t)\cdot t^{k_m/2}\rightarrow\frac{1}{\theta_{k_m}(x)\cdot(4\pi)^{k_m/2}} \ \ {\rm  uniformly\ on}\ \  (\Omega\cap \mathcal R_{k_m})\backslash E.$$
Hence, by using (\ref{equa4.14}) and the fact  that $\mu$ is  $k_m-$noncollapsing on $\Omega$ of scale $r_0$, we get
\begin{align}\label{equa4.15}
&\quad\ \lim_{t\to0}\int_{\Omega\cap \mathcal R_{k_m}} t^{k_m/2}\cdot H^\Omega(x,x,t)\du(x)\\
\notag &\ls C'_{N,K}\cdot\varepsilon+ \lim_{t\to0}\int_{(\Omega\cap \mathcal R_{k_m})\backslash E} t^{k_m/2}\cdot H^\Omega(x,x,t)\du(x)\\
\notag &= C'_{N,K}\cdot\varepsilon+\int_{(\Omega\cap \mathcal R_{k_m})\backslash E} \lim_{t\to0} t^{k_m/2}\cdot H^\Omega(x,x,t)\du(x) \\
\notag  &= C'_{N,K}\cdot\varepsilon+\int_{ \Omega\cap \mathcal R_{k_m} }\frac{1}{\theta_{k_m}(x)\cdot(4\pi)^{k_m/2}}\du(x)\\
\notag   &= C'_{N,K}\cdot\varepsilon+\int_{ \Omega\cap \mathcal R_{k_m} }\frac{\theta_{k_m}(x)}{\theta_{k_m}(x)\cdot(4\pi)^{k_m/2}}{\rm d}\mathscr H^{k_m}(x)\\
\notag    &= C'_{N,K}\cdot\varepsilon+\frac{\mathscr H^{k_m}( \Omega\cap \mathcal R_{k_m}) }{(4\pi)^{k_m/2}},
\end{align}
  where we have used $\mu|_{\mathcal R_{k_m}}=\theta_{k_m}\cdot\mathscr H^{k_m}.$ On the other hand, Fatou's lemma implies
\begin{equation*}
\begin{split}
 \lim_{t\to0}\int_{\Omega\cap \mathcal R_{k_m}} t^{k_m/2}\cdot H^\Omega(x,x,t)\du(x)&\gs  \int_{\Omega\cap \mathcal R_{k_m}} \lim_{t\to0} t^{k_m/2}\cdot H^\Omega(x,x,t)\du(x) \\
 &=\frac{\mathscr H^{k_m}( \Omega\cap \mathcal R_{k_m}) }{(4\pi)^{k_m/2}}.
\end{split}
\end{equation*}
Therefore, letting $\varepsilon\to0$, we conclude that
\begin{equation}\label{equa4.16}
 \lim_{t\to0}\int_{\Omega\cap \mathcal R_{k_m}} t^{k_m/2}\cdot H^\Omega(x,x,t)\du(x)=\frac{\mathscr H^{k_m}( \Omega\cap \mathcal R_{k_m}) }{(4\pi)^{k_m/2}}.
\end{equation}
For each $k<k_m$, From Lemma {\ref{lem4.4}}, we have, for $\mu$-almost all $x\in \mathcal R_k$, that
$$\lim_{t\to0}H^\Omega(x,x,t)\cdot t^{k_m/2}=\lim_{t\to0}H^\Omega(x,x,t)\cdot t^{ \frac{k}{2}}\cdot t^{\frac{k_m-k}{2}}=0,\quad \forall\ k<k_m.$$
By the same argument as deducing (\ref{equa4.15}), we obtain
\begin{equation}\label{eq4.15}
 \lim_{t\to0}\int_{\Omega\cap \mathcal R_{k}} t^{k_m/2}\cdot H^\Omega(x,x,t)\du(x)=0.
\end{equation}
The combination of (\ref{equa4.16}), (\ref{eq4.15}) and   Theorem \ref{thm4.2} (i)  implies  that
\begin{equation*}
\begin{split}
 \lim_{t\to0}t^{k_m/2}\sum_{j=1}^\infty e^{-\lambda_j^\Omega t}
 &= \lim_{t\to0}\int_\Omega t^{k_m/2}\cdot H^\Omega(x,x,t)\du  \\
&= \lim_{t\to0}\int_{\Omega \cap (\cup_{j=1}^{k_m}\mathcal R_j)}\!t^{k_m/2}\cdot H^\Omega(x,x,t)\du \\
 &= \lim_{t\to0}\sum_{j=1}^{k_m}\int_{\mathcal R_j\cap\Omega} t^{k_m/2} H^\Omega(x,x,t)\du =\frac{\mathscr H^{k_m}(\Omega\cap \mathcal R_{k_m}) }{(4\pi)^{k_m/2}},
\end{split}
\end{equation*}
 where, in the second equality, we have used the fact that $\mu(\Omega\cap\mathcal R_{j'})=0$ for any $j'>k_m.$

 Finally, by applying the Karamata Tauberian theorem, we have
 \begin{equation*}
 \lim_{\lambda\to\infty}\frac{N_\Omega(\lambda)}{\lambda^{k_m/2}} =\Gamma(k_m/2+1)^{-1}\cdot\frac{\mathscr H^{k_m}(\Omega\cap \mathcal R_{k_m}) }{(4\pi)^{k_m/2}}.
\end{equation*}
The proof is finished.
\end{proof}
\begin{rem}\label{rem4.7}
 We can consider  the constant $k_{\rm max}$ in (\ref{equa4.12})  as the maximal essential dimension of $\Omega$. It is clear that $k_{\rm max}\ls [N]$ for any domain $\Omega$ of an $RCD^*(K,N)$-space, by Theorem \ref{thm4.2}(i).
\end{rem}
\begin{proof}[Proof of Theorem \ref{thm1.1}.]
Let $\Omega\subset X$ be a bounded domain. We first show that
\begin{equation}\label{equa4.18}
\mu(\Omega\cap\mathcal R_k )=0,\quad{\rm for\ any\ integer}\ \  k<N.
\end{equation}
From Theorem \ref{thm4.2}(ii), we know that, for any integer $1\ls k\ls [N]$, $\mathcal R_k$ is $k$-rectifiable. Thus the Hausdorff dimension  $\dim_{\mathscr H}(\mathcal R_k)\ls k$. In particular, we have
 $$\mathscr H^N(\Omega\cap \mathcal R_k)=0,\quad{\rm for\ any\ integer}\ \  k<N.$$
  The assumption $\mu\ll \mathscr H^N$ implies $\mu(\Omega\cap\mathcal R_k )=0$. This is (\ref{equa4.18}).

  Secondly, we want to prove
 \begin{equation}\label{equa4.19}
N\in\mathbb N\quad {\rm and}\quad \mu\big(\Omega\backslash \mathcal R_N\big)=0.
\end{equation}
We will argue by a contradiction to show that $N$ is an integer. Suppose not, then we have $[N]<N$. From (\ref{equa4.18}), we get
 $\mu\big(\Omega\cap(\cup_{1\ls k\ls [N]}\mathcal R_k )\big)=0$. By  combining with Theorem \ref{thm4.2}(i), we conclude $\mu(\Omega)=0$. This contradicts to the fact that $\mu(O)>0$ for any open subset $O\subset X$ (since ${\rm supp}(\mu)=X$). Thus, $N$ is an integer. Now let us prove the assertion $\mu\big(\Omega\backslash \mathcal R_N\big)=0$.  By using (\ref{equa4.18}) again, we get
 $\mu\big(\Omega\cap(\cup_{1\ls k\ls N-1}\mathcal R_k )\big)=0$. It follows   $\mu\big(\Omega\backslash \mathcal R_N\big)=0,$ by Theorem \ref{thm4.2}(i).

At last, we will complete the proof of (\ref{equa1.1}). From (\ref{equa4.19}), we have
 $\mu\big(\Omega\cap \mathcal R_N\big)=\mu(\Omega)>0$, the definition of  $k_{\rm max}$ yields $k_{\rm max}\gs N$ (see (\ref{equa4.12})). Thus, we have $k_{\rm max}= N$ (recalling that $k_{\rm max}\ls N$, see Remark \ref{rem4.7}). Now we conclude by Theorem \ref{thm4.6} (and by \emph{Example 1}  above) that
\begin{align}\label{eq4.20}
	\lim_{\lambda\to\infty}\frac{N_\Omega(\lambda)}{\lambda^{N/2}}& =\Gamma(N/2+1)^{-1}\cdot\frac{\mathscr H^{N}(\Omega\cap \mathcal R_{N}) }{(4\pi)^{N/2}}\\
	\notag &= \frac{\omega_N\cdot \mathscr H^N(\Omega\cap \mathcal R_{N})}{(2\pi)^{N}},
 \end{align}
where we have used that $\Gamma(N/2)=\frac{2\pi^{N/2}}{N\cdot\omega_N}$.

We remain only to show that $\mathscr H^N(\Omega\cap \mathcal R_{N})=\mathscr H^N(\Omega)$. Equivalently,
$\mathscr H^N(\Omega \backslash\mathcal R_{N})=0.$
It follows immediately from the assumption $\mathscr H^N\ll \mu$ and (\ref{equa4.19}).
Now the proof is finished.
\end{proof}

 \appendix

\section{Dominated convergence theorem and Fatou's lemma on $pmGH$-converging spaces}

Dominated convergence theorem and Fatou's lemma are among the most important assertions in all of analysis. In this appendix, we will give an introduction of them for functions defined on a sequence of $pmGH$-converging metric measure spaces. They are well-known for experts.

Let $(X_j,d_j,\mu_j,p_j)$,   $j\in\mathbb N\cup\{\infty\},$ be a sequence of pointed metric measure spaces. In this appendix, we always assume that
\begin{align}\label{equ-a1}
&{\rm all\ of}\ \ (X_j,d_j)\ \ {\rm are\ length \ spaces}\\
\notag &{\rm and}\quad   (X_j,d_j,\mu_j,p_j)\overset{pmGH}{\longrightarrow} (X_\infty,d_\infty,\mu_\infty,p_\infty).
\end{align}
 Please see Proposition \ref{prop2.9} for the definitions of pointed measured Gromov-Hausdorff convergence ($pmGH$-convergence).

\begin{defn}\label{def-a1}
 Let $R>0$. Let $\{f_j\}$ be a sequence of  Borel functions defined on $B_R(p_j)$ for each $j\in\mathbb N\cup\{\infty\}$.
We say that
 \begin{equation}\label{equ-a2}
\liminf_{j\to\infty}f_j\gs f_\infty\ \ {\rm at}\ \ x_\infty \in B_R(p_\infty)
\end{equation}
 if $\liminf_{j\to\infty}f_j(x_j)\gs f_\infty(x_\infty)$ holds  for any sequence $\{x_j\}_{j\in\mathbb N}$ converging to $x_\infty$.
  More precisely, by letting
$(\Phi_j)$ and $(\epsilon_j)$ be given in Proposition \ref{prop2.9}, (\ref{equ-a2}) means the following:  for any $\varepsilon>0$, there exist $N(\varepsilon,x_\infty)\in\mathbb N$ and $\delta(\varepsilon, x_\infty)>0$ such that
\begin{equation}\label{equ-a3}
\inf_{z\in B_R(p_j),\ d_\infty\big(\Phi_j(z),x_\infty\big)\ls \delta(\varepsilon,x_\infty)} f_j(z)\gs f_\infty(x_\infty)-\varepsilon.
\end{equation}
\end{defn}

We say that $\limsup_{j\to\infty}f_j\ls f_\infty$ at $x_\infty \in B_R(p_\infty)$ if and only if $$\liminf_{j\to\infty}(-f_j)\gs -f_\infty$$ at $x_\infty$. It is clear that  $f_j\rightarrow f_\infty$ over $B_R(p_j)$ at  $x_\infty\in B_R(p_\infty)$ in the sense of Definition \ref{definition2.10} $(i)$ if and only if $\liminf_{j\to\infty}f_j\gs f_\infty$  and  $\limsup_{j\to\infty}f_j\ls f_\infty$ at $x_\infty \in B_R(p_\infty)$.

\begin{prop}[Fatou's Lemma  on $pmGH$-converging spaces]\label{prop-a2}
Let $R>0$.
Let  $\{f_j\}_{j\in\mathbb N\cup\{\infty\}}$ be a sequence of  nonnegative Borel real function  defined on $B_R(p_j)$. Suppose that $f_\infty\in L^1(B_R(p_\infty))$ and lower semi-continuous $\mu_\infty$-a.e. on $ B_R(p_\infty)$, and that
\begin{equation}\label{equ-a4}
\liminf_{j\to\infty}f_j\gs f_\infty,\quad  \ \mu_\infty{\rm-a.e.\ in}\   B_R(p_\infty),
\end{equation}
then
\begin{equation}\label{equ-a5}
 \liminf_{j\to\infty} \int_{B_R(p_j)}f_j{\rm d}\mu_j\gs \int_{B_R(p_\infty)}f_\infty{\rm d}\mu_\infty.
 \end{equation}
\end{prop}

We need the following   a variant of the classical Fatou's lemma:
\begin{lem}\label{lem-a3}
Let $(X,d)$ be a metric  space, and let $g$ be a nonnegative Borel real function  on $X$. Suppose that  $\{\nu_j\}_{j\in\mathbb N\cup\{\infty\}} $ be a sequence of Radon measures   on $X$ such that   $\nu_j\rightharpoonup\nu_\infty$, as $j\to\infty$. Assume that $g$ is lower semi-continuous $\nu_\infty$-a.e. on $X$. Then we have
$$\liminf_{j\to\infty}\int_X g{\rm d}\nu_j\gs \int_X g{\rm d\nu_\infty}.$$
\end{lem}
\begin{proof}
Since $g$ is lower semi-continuous almost everywhere, there exist a sequence of (Lipschitz) continuous functions $g_t$ such that $g_t(x)\ls g(x)$ and $g_t(x)\uparrow g(x)$ as $t\to\infty$ at $\nu_\infty$-a.e. $x\in X$ (\cite[Lemma 1.61]{afp06}).

Fix each $t>0$, we put, for any $s\gs0$, that
$$E_t(s):=g_t^{-1}\big((s,\infty)\big)\quad {\rm and}\quad G_j(s):=\nu_j(E_t(s)).$$
Then $G_j\gs0$ and $E_t(s)$ is open, and by $\nu_j\rightharpoonup\nu_\infty$ that $\lim_{j\to\infty}G_j(s)\gs G_\infty(s)$. By using the fact that
$\int_Xg_t{\rm d}\nu_j=\int_0^\infty G_j(s){\rm d}s$ ant the Fatou's lemma on $[0,\infty)$, we conclude  that
$$\liminf_{j\to\infty}\int_X g_t{\rm d}\nu_j\gs \int_X g_t{\rm d\nu_\infty}, \quad {\rm for\ all\ }\ t>0.$$
At last, the assertion comes from the fact $g_t\ls g$ and the monotone converge theorem.
\end{proof}

\begin{proof}[Proof of Proposition \ref{prop-a2}]
Let $(\Phi_j)$ and $(\epsilon_j)$ be given in Proposition \ref{prop2.9}.

For any fixed $\varepsilon>0$ and $k\in\mathbb N$, we denote
\begin{align*}
	A_{k}(\varepsilon):=\Big\{& x_\infty\in B_R(p_\infty):\   \forall\ \ell\gs k, \ {\rm it\ holds}\\
	&\inf_{z\in B_R(p_\ell),\ d_\infty\big(\Phi_\ell(z),x_\infty\big)\ls \delta(\varepsilon,x_\infty)} f_\ell(z)>f_\infty(x_\infty)-\varepsilon\Big\}.
\end{align*}
It is easily seen that $ A_k(\varepsilon)$ is increasing in  $k$ and by (\ref{equ-a4}) that
 $$\mu_\infty(B_R(p_\infty))=\mu_\infty\big(\cup_{k\gs1}A_k(\varepsilon)\big)=\lim_{k\to\infty}\mu_\infty(A_k(\varepsilon)).$$
  Then, for any fixed $\varepsilon>0$, there exists some $k_0:=k_0(\varepsilon)\in\mathbb N$ such that
\begin{equation}\label{equ-a6}
\int_{A_{k_0}(\varepsilon)}f_\infty{\rm d}\mu_\infty\gs  \int_{B_R(p_\infty)}f_\infty{\rm d}\mu_\infty -\varepsilon,
\end{equation}
where we have used $f_\infty\in L^1(B_R(p_\infty))$.
We put
$$E_j(s):=\Big\{x\in B_R(p_j):\ f_j(x)>s\Big\},\quad \forall\ j\in\mathbb N\cup\{\infty\},\quad\forall\ s\in[0,\infty).$$
Given any $\varepsilon>0$ (and fixed some $k_0$ in (\ref{equ-a6})),
   we have
\begin{equation}\label{equ-a7}
\Phi_\ell^{-1}\big(E_\infty(s)\cap A_{k_0}(\varepsilon)\big) \subset   E_\ell(s-\varepsilon),\quad \forall\ \ell\gs k_0, \quad\forall\ s\in[\varepsilon,\infty).
\end{equation}
Indeed, letting $\ell\gs k_0$, for each $x\in \Phi_\ell^{-1}\big(E_\infty(s)\cap A_{k_0}(\varepsilon)\big) $, we get $\Phi_\ell(x)\in E_\infty(s) \cap A_{k_0}(\varepsilon)$. This implies   $f_\infty(\Phi_\ell(x))>s $ and
    $f_\ell(x)>f_\infty(\Phi_\ell(x))-\varepsilon$. It follows
     $f_\ell(x)>s-\varepsilon$. I.e., $x\in E_\ell(s-\varepsilon)$.

By (\ref{equ-a7}), we have
$$\mu_\ell\big(E_\ell(s-\varepsilon)\big)\gs \mu_\ell\big(\Phi_\ell^{-1}\big(E_\infty(s)\cap A_{k_0}(\varepsilon)\big)\big)=[(\Phi_\ell)_\sharp\mu_\ell](E_\infty(s)\cap A_{k_0}(\varepsilon))$$
for any $\ell\gs k_0$ and any $s>\varepsilon$. By integrating over $s\in(\varepsilon, \infty)$, we get
$$\int_{B_R(p_\ell)}f_\ell{\rm d}\mu_\ell=\int_\varepsilon^\infty\mu_\ell\big(E_\ell(s-\varepsilon)\big){\rm d}s\gs \int_{A_{k_0}(\varepsilon)\cap\{f_\infty\gs \varepsilon\}}f_\infty{\rm d}[(\Phi_\ell)_\sharp\mu_\ell]$$
for any $\ell\gs k_0$. Thus, by  the  weak convergence  $(\Phi_\ell)_\sharp\mu_\ell \rightharpoonup\mu_\infty$ on $B_R(p_\infty)$ and Lemma \ref{lem-a3}, we conclude that, for any fixed $\varepsilon>0$,
\begin{align}\label{equ-a8}
	\liminf_{\ell\to\infty}\int_{B_R(p_\ell)}f_\ell{\rm d}\mu_\ell &\gs \int_{A_{k_0}(\varepsilon)\cap\{f_\infty\gs \varepsilon\}}f_\infty{\rm d}\mu_\infty \\
	\notag &\gs  \int_{A_{k_0}(\varepsilon)}f_\infty{\rm d}\mu_\infty-  \int_{B_R(p_\infty)\cap \{f_\infty<\varepsilon\}}f_\infty{\rm d}\mu_\infty  .
\end{align}
At last, letting $\varepsilon\to0$ and noting that $f_\infty\in L^1((B_R(p_\infty))$, the assertion (\ref{equ-a5}) comes from the combination of (\ref{equ-a6}) and (\ref{equ-a8}). The proof is finished.
\end{proof}
From this, it is not hard to deduce the dominated converge theorem for functions living on a sequence of $pmGH$-converging spaces as following.
\begin{prop}[Dominated convergence thoerem  on $pmGH$-\ converging spaces]\label{prop-a4}
Let $R>0$  and let  $\{f_j\}_{j\in\mathbb N\cup\{\infty\}}$ be a sequence of  Borel real function  defined on $B_R(p_j)$. Suppose that $f_j\to f_\infty$  $\mu_\infty$-a.e. on $B_R(p_\infty).$ If
 there exists a sequence of functions $\{F_j\}_{j\in\mathbb N\cup\{\infty\}}$ such that $F_j\to F_\infty$ in $L^1(B_R(p_j))$ in the sense of Definition \ref{def3.1} (by replacing $L^2$ there by $L^1$), and that
$$|f_j(x)| \ls F_j(x)\qquad\forall  x\in B_R(p_j),\  \forall j\in  \mathbb N,$$
and  $ |f_\infty | \ls F_\infty  $ for $\mu_\infty$-almost all  in $B_R(p_\infty)$,    then
\begin{equation}\label{equ-a10}
 \lim_{j\to\infty} \int_{B_R(p_j)}|f_j|{\rm d}\mu_j= \int_{B_R(p_\infty)}|f_\infty|{\rm d}\mu_\infty.
 \end{equation}
\end{prop}
\begin{proof} From Remark \ref{remark2.11}(2), we know that $f_\infty$ is continuous at almost everywhere. By using   Proposition \ref{prop-a2} to both $|f_j|$ and $F_j-|f_j|$, the assertion follows.
\end{proof}



\end{document}